\documentclass[12pt,twoside]{article}
\usepackage[english]{babel}
\usepackage[latin1]{inputenc}
\usepackage{amsmath}
\usepackage{amsthm}
\usepackage{amssymb,amsfonts}
\usepackage{comment}
\usepackage{tikz-cd}
\usepackage{tikz}
 \usetikzlibrary {shapes.geometric} 
\usepackage{orcidlink}
\usepackage{stmaryrd}
\usepackage{pst-node}
\usepackage{caption}
\usepackage{array}
\usepackage[all]{xy}

\usepackage{graphicx}                   %% for include pictures *.eps
\usepackage{hyperref}

\newcommand{\CP}[1]{\mathbb{C}P^{#1}}      %complex projective spaces

\begin{document}
\raggedbottom
\pagestyle{myheadings}
%\markboth{\centerline{...}}
%\markboth{Mikl\'os Eper, Szil\'ard Szab\'o}{Geometric P=W conjecture in rank three case \dots}
\title
{\textsc{Rank three representations of Painlev\'e systems: III. Dolbeault structure, spectral correspondence}}

\author{by \textsc{Mikl\'os Eper}\footnote{\textsl{Department of Algebra and Geometry, Institute of Mathematics, Faculty of Natural Sciences, Budapest University of Technology and
Economics, M\H uegyetem rkp. 3., Budapest H-1111, Hungary}, e-mail: \href{mailto:epermiklos@gmail.com}{epermiklos@gmail.com}} and \textsc{Szil\'ard Szab\'o}\footnote{Institute of Mathematics, Faculty of Science, E\"otv\"os Lor\'and University, P\'azm\'any P\'eter s\'et\'any 1/C, Budapest, Hungary, H-1117; Alfr\'ed R\'enyi Institute of Mathematics,
Re\'altanoda utca 13-15., Budapest 1053, Hungary, e-mails:\href{mailto:szilard.szabo@ttk.elte.hu}{szilard.szabo@ttk.elte.hu} and \href{mailto:szabo.szilard@renyi.hu}{szabo.szilard@renyi.hu}}
\orcidlink{https://orcid.org/my-orcid?orcid=0009-0008-7435-2021}}

\begin{comment}
\author[1]{\textsc{Mikl\'os Eper}}
\author[2]{\textsc{Szil\'ard Szab\'o}}
\affil[1]{\textsl{(Department of Algebra and Geometry, Institute of Mathematics, Faculty of Natural Sciences, Budapest University of Technology and
Economics, M\H uegyetem rkp. 3., Budapest H-1111, Hungary)}}
\affil[2]{\textsl{(Department of Algebra and Geometry, Institute of Mathematics, Faculty of Natural Sciences, Budapest University of Technology and
Economics, M\H uegyetem rkp. 3., Budapest H-1111, Hungary; Alfr\'ed R\'enyi Institute of Mathematics,
Re\'altanoda utca 13-15., Budapest 1053, Hungary)}}
\end{comment}
\maketitle

%%%%%%%%%%%%%%%%%%%%%%%%%%%%%%%%%%%%%%%%%%%%

\begin{abstract}
We prove that there exists a holomorphic symplectic isomorphism between the rank $2$ and $3$ representations of the Painlev\'e systems in the Dolbeault complex structure, and give explicit descriptions of the corresponding elliptic fibrations. 
This, combined with the de Rham description given in~\cite{ESz3}, implies that the corresponding moduli spaces are hyperK\"ahler isometric to each other. 
\end{abstract}

%%%%%%%%%%%%%%%%%%%%%%%%%%%%%%%%%%%%%%%%%%%
\newtheorem{theorem}{Theorem}[section]
\newtheorem{corollary}[theorem]{Corollary}
\newtheorem{conjecture}{Conjecture}[section]
\newtheorem{lemma}[theorem]{Lemma}
\newtheorem{exmple}[theorem]{Example}
\newtheorem{defn}[theorem]{Definition}
\newtheorem{prop}[theorem]{Proposition}
\newtheorem{rmrk}[theorem]{Remark}
\newtheorem{claim}[theorem]{Claim}

            %%% for no-italic, numbered environments, use:
\newenvironment{definition}{\begin{defn}\normalfont}{\end{defn}}
\newenvironment{remark}{\begin{rmrk}\normalfont}{\end{rmrk}}
\newenvironment{example}{\begin{example}\normalfont}{\end{example}}
            %%% for unnumbered environments, use f.e.
\newenvironment{acknowledgement}{{\bf Acknowledgement:}}

\newcommand\restr[2]{{% we make the whole thing an ordinary symbol
  \left.\kern-\nulldelimiterspace % automatically resize the bar with \right
  #1 
  \vphantom{\big|} % pretend it's a little taller at normal size
  \right|_{#2} % this is the delimiter
  }}

%%%%%%%%%%%%%%%%%%%%%%%%%%%%%%%%%%%%%%%%%%%%%%%%%%%%%%%%%%

\section{Introduction}

In the first one of this series of articles~\cite{ESz2}, we have studied the wild character varieties associated to the rank $3$ Lax representations of the Painlev\'e equations found by Joshi, Kitaev and Treharne (JKT)~\cite{JKTI},~\cite{JKTII}. 
Then, in the second article~\cite{ESz2}, we have studied the properties of Fourier--Laplace transformation $\mathcal{F}$ for the corresponding $\mathcal{D}$-modules. 
The aim of this article is to offer a comprehensive description of the same correspondence, this time between the Dolbeault moduli spaces, in terms of spectral data and algebraic Nahm transformation. 
Our main result is: 
\begin{theorem}\label{thm:Dolisometry}
      Algebraic Nahm transformation establishes an algebraic symplectic isomorphism with respect to the Dolbeault complex structures between each $\mathcal{M}_{\textrm{Dol}}^{JKT*}$ and the corresponding $\mathcal{M}_{\textrm{Dol}}^{P*}$. 
\end{theorem}
On the other hand, in~\cite{ESz3} we have proved that Fourier--Laplace transformation of $\mathcal{D}$-modules determines an algebraic isomorphism between corresponding de Rham moduli spaces of irregular connections with given normal forms. 
In Proposition~\ref{prop:FLN} we establish that Nahm and Fourier--Laplace transformations agree, up to applying the non-abelian Hodge diffeomorphism. 
Combining these results, we then immediately get: 
\begin{theorem}\label{thm:FLhK}
Fourier--Laplace transformation $\mathcal{F}$ of $\mathcal{D}$-modules induces a hyperK\"ahler isometry between the moduli spaces of rank $2$ irregular connections associated to the Painlev\'e equations and the corresponding JKT moduli spaces of irregular connections of rank $3$.  
\end{theorem}
The second author proved in~\cite{Sz_Plancherel} that if the singularities away from $\infty$ are regular and the irregular singularity at $\infty$ is of the simplest possible type (unramified of Katz invariant $1$, c.f. the notion of a non-commutative Hodge structure of exponential type~\cite[Definition~2.12]{KKP}), then Fourier--Laplace (also known as Nahm) transformation acts as a hyperK\"ahler isometry between the complete moduli spaces. 
This result applies directly to one of the cases that we consider here (Painlev\'e VI), and the idea of the proof carries over too, up to some technical details. 
Namely, we need to work out the deformation theory of the associated Dolbeault moduli spaces. 
We also identify the Dolbeault moduli spaces as the integrable systems in terms of specific elliptic pencils, and we determine their fiber at infinity (Section~\ref{sec:Dol}). 
We find that they agree with the fibers at infinity found for the Painlev\'e cases~\cite{Sz_Painl}. 
The main technical points of the proof are Theorem~\ref{thm:BNR}, where we show that the Dolbeault moduli space is birational to a certain Mukai moduli space of sheaves on the total space of a suitable elliptic fibration (which can be considered as a special case of~\cite{BNR}), and Proposition~\ref{prop:symplectic} where we show that this map preserves the natural symplectic structures. 

We expect that Fourier--Laplace--Nahm transformation is a hyperK\"ahler isometry between suitable non-abelian Hodge moduli spaces in higher generality, too. 
Up to date, as far as we are aware, the only proofs of nontrivial hyperK\"ahler isometry between moduli spaces of solutions of the self-duality equations use one version or another of Nahm's transformation: for the ADHM-transformation see~\cite{Mac}, for dual $4$-tori~\cite{BvB}, for monopoles~\cite{Nak}, for doubly-periodic instantons~\cite{BJ}, for instantons on multi-Taub-NUT space~\cite{CLHS}. 
We note that in all the listed examples, isometry of Nahm's transformation follows from tedious computations using harmonic spinors, while our proof relies on a more geometric approach. 
We believe that our current, geometric proof should carry over to more general situations. 

\section{Preliminaries}

\subsection{Higgs bundles with higher order poles}\label{higgs}

Throughout the paper, we will work over the base field $\mathbb{C}$. 
Let $X$ be a smooth projective curve and $D=\sum m_ip_i$ an effective divisor on $X$, not necessarily reduced. 
A meromorphic Higgs bundle of rank $r$ over $(X,D)$ is a couple $(\mathcal{E},\theta )$ where $\mathcal{E}$ is a holomorphic vector bundle of rank $r$ over $X$ and $\theta \in H^0 (X, End(\mathcal{E}) \otimes K_X(D))$ where $K_X$ stands for the canonical bundle of $X$. 
In this paper we will be mostly dealing with the cases $r=2,3$. 
Two meromorphic Higgs bundles $(\mathcal{E}_i ,\theta_ i )$ with $i=1,2$ are said to be isomorphic to each other if there exists a bundle isomorphism 
\[
    \Phi\colon \mathcal{E}_1 \to \mathcal{E}_2
\] 
such that 
\[
    \theta_2 \circ \Phi = (\Phi \otimes \operatorname{I}_{K_X(D)}) \circ \theta_1 ,
\]
where $\operatorname{I}_{K_X(D)}$ is the identity automorphism of $K_X(D)$. 
By~\cite{Mark}, there exists a moduli space of isomorphism classes of stable meromorphic Higgs bundles of given rank $r$ and degree $d$ over $X$, and it is a holomorphic Poisson manifold. 
Moreover, the symplectic leaves of this Poisson manifold are given by fixing the irregular part and residue of $\theta$ at $D$. 
The cases when these symplectic leaves are of the lowest possible complex dimension $2$, have been listed in~\cite[Section~4.1]{Boa3}. 
The case of interest to us is the one related to the root system $\widetilde{E}_6$, that has numerical data $r=3$ and $\operatorname{length}(D) = 3$. 
Three different partitions realize the possible divisors $D$: $(1,1,1), (2,1)$ and $3$. 
The first one will be referred to as the logarithmic case, and the other two as the irregular ones. 
In both irregular cases, the Jordan type of the leading-order coefficient of $\theta$ may belong to three combinatorially different adjoint orbits
\begin{enumerate}
    \item semisimple orbits,
    \item orbits with a Jordan block of size $2$ and a Jordan block of size $1$,
    \item and finally orbits with a Jordan block of size $3$.  
\end{enumerate} 
We will call these cases respectively the unramified, minimally ramified and totally (or maximally) ramified (or twisted) cases. 
We will always assume that the eigenvalues corresponding to different Jordan blocks of the leading-order term are different, said differently that it is regular.
This gives us a total of $7$ combinatorial cases to study. 

\subsubsection{Holomorphic symplectic leaves}

From now on, we will set $X = \CP1$, covered by two affine sets $X =\operatorname{Spec}\mathbb{C}[z] \cup \operatorname{Spec}\mathbb{C}[w]$ with $w = z^{-1}$. 
We will denote by $\infty$ the closed point $(w) \in \operatorname{Spec}\mathbb{C}[w]$, that we take to be in $D$. 
The aim of this section is to define ${\mathcal{M}}_{Dol}^{JKT*}$, the moduli space of rank 3 meromorphic, parabolic, $\alpha$-stable Higgs bundles over $\mathbb{C}P^1$. 
Here, the symbol $*$ refers to the one of the following combinatorial possibilities, encoding the number of singular points and the formal types of the corresponding differential module at the irregular singularity. %introduced in Section~\ref{higgs}. 
Namely, $*$ is one of the following possibilities: 
\begin{equation*}
    \{VI,V,IVa,IVb,II,I\}, 
\end{equation*}
That is, $JKT*$ refers to the systems described in \cite{JKTI}, \cite{JKTII}. In cases $JKTVI$, $JKTV$, $JKTIVa$ there are two singularities, one logarithmic and one where $\theta$ admits a pole of order two, and this latter is untwisted, minimally twisted and maximally twisted respectively. 
In the cases  $JKTIVb$, $JKTII$, $JKTI$, $\theta$ has one singularity of order three, which is untwisted, minimally twisted and maximally twisted respectively. 
In each of these cases, the multiplicity of $\{ \infty \}$ in $D$ is at least $2$, i.e. the corresponding point is an irregular singularity. 
There is also a seventh case, which we will refer to as the logarithmic case, with three logarithmic singularity. 
This latter case was the object of our previous paper~\cite{ESz}. 
We will spell out the local models in more detail in equations below. 

Consider the holomorphic coordinate chart $z-p_i$ centered at $p_i$, and take a local trivialization of $\mathcal{E}$. 
Then, the Higgs field has the Laurent-series expansion:
\begin{equation}
    \label{laurent}
    \theta =\sum_{k=-m_i}^{\infty} A_k(z-p_i)^{k}\otimes\textrm{d}z, 
\end{equation}
where $A_k\in\mathfrak{gl}(3,\mathbb{C})$, and $\operatorname{Res}_{p_i}\theta:=A_{-1}$ is the residue, which is well-defined up to the GL$(3,\mathbb{C})$-adjoint action. 
The integer $m_i$ is the order of the pole.
The singularity is said to be logarithmic in case $m_i=1$.
For generic choices of the parameters (see below) the pole is irregular in case $m_i>1$. 

In our setup, we fix the following data:
\begin{itemize}
    \item The divisor $D$ on $\mathbb{C}P^1$ (the points $p_i$'s and numbers $m_i$'s).
    \item The irregular part and residue of $\theta$ at $D$.
    \item parabolic weights $\{\alpha_{p_i}^j\}$ at $D$ (the same way, as in the logarithmic case in \cite{ESz}).
\end{itemize}
For the importance of parabolic weights, see the proof of Theorem~\ref{thm:moduli}. 
We make the assumption 
\begin{equation}\label{eq:par_wt}
    0 < \alpha_{p_i}^j < 1.
\end{equation}
These data will then define the Dolbeault moduli space ${\mathcal{M}}_{Dol}^{JKT*}$, and the exact choice of the data decides which of the 6 cases the symbol $*$ denotes. 

Specifically, the local forms at $\infty$ of the meromorphic Higgs fields are fixed as follows, with respect to some holomorphic trivialization of $\mathcal{E}$. In the cases $JKTVI, JKTV, JKTIVa$ corresponding to $D=\{0\}+2\cdot \{\infty\}$, the eigenvalues of $\operatorname{Res}_0 \theta$ are fixed generic values, denoted by $\tau_i$, $0 \leq i \leq 2$.

\paragraph{JKTVI}
We have $D = 2 \cdot \{ \infty \} + \{ 0 \}$, such that the irregular singularity $p_0=\infty$ is untwisted, with local form at $p_0=\infty$ given by 
\begin{equation}
\label{polar1}\tag{JKTVI}
    \theta=\left[
    \begin{pmatrix}
        a_0 & 0 & 0 \\
        0 & a_1 & 0 \\
        0 & 0 & a_2
    \end{pmatrix}w^{-2}+\begin{pmatrix}
        b_0 & 0 & 0 \\
        0 & b_1 & 0 \\
        0 & 0 & b_2
    \end{pmatrix}w^{-1}+O(1)
    \right]\otimes\textrm{d}w,
\end{equation}
where $a_i,b_i\in\mathbb{C}$ ($i=0,1,2$) are fixed, $a_i$'s are mutually different, and 
\[
b_0+b_1+b_2=\operatorname{Tr Res}_{0} \theta
\]
because of the residue theorem.

\paragraph{JKTV}
We have $D = 2 \cdot \{ \infty \} + \{ 0 \}$, such that the irregular singularity $p_0=\infty$ is minimally twisted, with local form at $p_0=\infty$ given by 
\begin{equation}
\label{polar2}\tag{JKTV}
    \theta=\left[
    \begin{pmatrix}
        a_0 & 1 & 0 \\
        0 & a_0 & 0 \\
        0 & 0 & a_1
    \end{pmatrix}w^{-2}+\begin{pmatrix}
        0 & 0 & 0 \\
        b_0 & b_1 & 0 \\
        0 & 0 & b_2
    \end{pmatrix}w^{-1}+O(1)
    \right]\otimes\textrm{d}w,
\end{equation}
where $a_0,a_1,b_i\in\mathbb{C}$ ($i=0,1,2$) are fixed, $b_0 \neq 0$, $a_0,a_1$ are different, and 
\[
b_1+b_2=\operatorname{Tr Res}_{0} \theta.
\] 

\paragraph{JKTIVa}
We have $D = 2 \cdot \{ \infty \} + \{ 0 \}$, such that the irregular singularity $p_0=\infty$ is maximally twisted, with local form at $p_0=\infty$ given by 
\begin{equation}
\label{polar3}\tag{JKTIVa}
    \theta=\left[
    \begin{pmatrix}
        a_0 & 1 & 0 \\
        0 & a_0 & 1 \\
        0 & 0 & a_0
    \end{pmatrix}w^{-2}+\begin{pmatrix}
        0 & 0 & 0 \\
        0 & 0 & 0 \\
        b_0 & b_1 & b_2
    \end{pmatrix}w^{-1}+O(1)
    \right]\otimes\textrm{d}w,
\end{equation}
where $a_0,b_i\in\mathbb{C}$ ($i=0,1,2$) are fixed, and $b_2=\operatorname{Tr Res}_{0} \theta$.

\paragraph{JKTIVb}
We have $D = 3 \cdot \{ \infty \}$, such that the irregular singularity $p_0=\infty$ is untwisted, with local form at $p_0=\infty$ given by 
\begin{equation}
\label{polar4}\tag{JKTIVb}
    \theta=\left[
    \begin{pmatrix}
        a_0 & 0 & 0 \\
        0 & a_1 & 0 \\
        0 & 0 & a_2
    \end{pmatrix}w^{-3}+\begin{pmatrix}
        b_0 & 0 & 0 \\
        0 & b_1 & 0 \\
        0 & 0 & b_2
    \end{pmatrix}w^{-2}+\begin{pmatrix}
        c_0 & 0 & 0 \\
        0 & c_1 & 0 \\
        0 & 0 & c_2
    \end{pmatrix}w^{-1}+O(1)
    \right]\otimes\textrm{d}w,
\end{equation}
where $a_i,b_i,c_i\in\mathbb{C}$ ($i=0,1,2$) are fixed, $a_i$'s are mutually different, and $c_0+c_1+c_2=0$, again because of the residue theorem.

\paragraph{JKTII}
We have $D = 3 \cdot \{ \infty \}$, such that the irregular singularity $p_0=\infty$ is minimally twisted, with local form at $p_0=\infty$ given by  
\begin{equation}
\label{polar5}\tag{JKTII}
    \theta=\left[
    \begin{pmatrix}
        a_0 & 1 & 0 \\
        0 & a_0 & 0 \\
        0 & 0 & a_1
    \end{pmatrix}w^{-3}+\begin{pmatrix}
        0 & 0 & 0 \\
        b_0 & b_1 & 0 \\
        0 & 0 & b_2
    \end{pmatrix}w^{-2}+\begin{pmatrix}
        0 & 0 & 0 \\
        c_0 & c_1 & 0 \\
        0 & 0 & c_2
    \end{pmatrix}w^{-1}+O(1)
    \right]\otimes\textrm{d}w,
\end{equation}
where $a_0,a_1,b_i\in\mathbb{C}$ ($i=0,1,2$) are fixed, $a_0,a_1$ are different, and $c_1+c_2=0$.

\paragraph{JKTI}
We have $D = 3 \cdot \{ \infty \}$, such that the irregular singularity $p_0=\infty$ is maximally twisted, with local form at $p_0=\infty$ given by 
\begin{equation}
\label{polar6}\tag{JKTI}
    \theta=\left[
    \begin{pmatrix}
        a_0 & 1 & 0 \\
        0 & a_0 & 1 \\
        0 & 0 & a_0
    \end{pmatrix}w^{-3}+\begin{pmatrix}
        0 & 0 & 0 \\
        0 & 0 & 0 \\
        b_0 & b_1 & b_2
    \end{pmatrix}w^{-2}+\begin{pmatrix}
        0 & 0 & 0 \\
        0 & 0 & 0 \\
        c_0 & c_1 & 0
    \end{pmatrix}w^{-1}+O(1)
    \right]\otimes\textrm{d}w,
\end{equation}
where $a_0,b_i,c_i\in\mathbb{C}$ ($i=0,1,2$) are fixed. 

\begin{theorem}~\cite{BB}\label{thm:moduli}
    For fixed numerical invariants $a_i, b_i, c_i$ as appropriate, there exist moduli spaces ${\mathcal{M}}_{Dol}^{JKT*}$ of gauge equivalence classes of $\alpha$-stable meromorphic Higgs bundles having the above local forms. 
    They are real $4$-dimensional complete hyperK\"ahler manifolds. 
    Moreover, endowed with another complex structure $J$ of the hyperK\"ahler family, the spaces ${\mathcal{M}}_{Dol}^{JKT*}$ are $\mathbb{C}$-analytically isomorphic to moduli spaces ${\mathcal{M}}_{dR}^{JKT*}$ of $\beta$-stable integrable connections with irregular singularities of fixed irregular types and residues. 
\end{theorem}

For the notion of irregular type of an integrable connections with irregular singularities, see~\cite{Boa4} or its summary in~\cite[Section~2.2.1]{ESz2}. 
There results a diffeomorphism
\begin{equation}\label{eq:NAH}
    NAH\colon {\mathcal{M}}_{Dol}^{JKT*} \longrightarrow {\mathcal{M}}_{dR}^{JKT*}. 
\end{equation}

\begin{rmrk}
    The strategy of the below proof works in higher generality too, once suitable local forms are fixed. 
    (Obviously, the dimension must be suitably modified for the generalization.)
\end{rmrk}

\begin{proof}
In the untwisted cases, all results except for the dimension count are contained in~\cite[Theorem~0.2]{BB}. 
In the twisted cases, one needs to apply an equivariant version of the theory of~\cite{BB}. 
An equivariant wild harmonic metric has been found in~\cite[Section~13.4]{Moc}.
However, the existence of the moduli spaces has not been addressed in \emph{loc. cit.}, so we now spell out some of the details. 
We will not give full details, rather we content ourselves with explaining the new ideas necessary for reducing the construction to the one of~\cite{BB}. 

Namely, a Hermitian metric in $\mathcal{E}$ is said to be Hermitian--Einstein for $(\mathcal{E}, \theta )$ if the associated connection 
\[
    D_h^+ + \theta + \theta^{\dagger_h} 
\]
is integrable on $X\setminus D$. 
Here, $D_h^+ = \bar{\partial}_{\mathcal{E}} + \partial_h$ stands for the Chern connection 
associated to $\bar{\partial}_{\mathcal{E}}$ and $h$, and $\dagger_h$ stands for 
Hermitian transpose with respect to $h$. 
If $h$ is Hermitian--Einstein for $(\mathcal{E}, \theta )$, then the triple $(\mathcal{E}, \theta , h)$ is said to be a harmonic bundle on $X\setminus D$. 
A rank $r$ harmonic bundle is called unramifiedly (good)\footnote{In the case of complex curves, the adjective "good" is superfluous, so we will drop it in this discussion.} wild~\cite[Section~1.2]{Moc} if near every singular point $p\in X$, with respect to some local holomorphic coordinate centered at $p$ (that is $t=z-p$, if $p\in\operatorname{Spec}\mathbb{C}[z]$ and $t=w$, if $p=\infty$) and some holomorphic trivialization of $\mathcal{E}$, we have 
\begin{enumerate}
    \item a decomposition 
\begin{equation}
\label{eq:decomposition}
    (\mathcal{E}, \theta ) = \bigoplus_i  
\left( \mathbb{C}( t ) , \frac 12 \operatorname{d}\! q_i + \nu_i \frac{\operatorname{d}\! t}t \right) \otimes \mathcal{R}_i
\end{equation}
with
\begin{itemize}
    \item $q_i \in t^{-1}\mathbb{C}[t^{-1}]$ some Laurent polynomials without constant term,
    \item some $\nu_i \in \mathbb{C}$ (eigenvalues of the residue), and 
    \item some Higgs bundle 
\[
    \mathcal{R}_i = \left( \mathbb{C}( t )^{\oplus r_i}, N_i \frac{\operatorname{d}\! t}t + O(1) \operatorname{d}\! t \right)
\]
where $N_i \in \mathfrak{gl}_{r_i}(\mathbb{C})$ is some nilpotent endomorphism (where, obviously $r = \sum_i r_i$ is some partition); 
\end{itemize}
\item as $t\to 0$, the Hermitian metric $h$ is bounded with respect to 
\begin{equation}\label{eq:model_metric}
    h_0 = \operatorname{diag} (|t|^{2 \alpha_i} |\ln |t|^2|^k), 
\end{equation}
where 
\begin{itemize}
\item 
\[
    0 \leq \alpha_1 \leq \alpha_2 \leq \cdots < 1
\]
are some numbers (parabolic weights)\footnote{The arguments of this theorem hold without the positivity assumption~\eqref{eq:par_wt}.},
\item $N_i$ respects the $\alpha$-filtration: if $\alpha_i < \alpha_{i'}$ and $e_i, e_{i'}$ are corresponding basis vectors in the diagonal decomposition of $h_0$ then the coefficient of $e_{i'}$ in the expression of $N e_i$ as linear combination of the basis vectors is $0$, 
\item $k$ is the weight of $N_i$, namely $k=2j-1-n$ on $\operatorname{ker}(N_i^j)- \operatorname{ker}(N_i^{j-1})$. 
\end{itemize}
\end{enumerate}

A wild harmonic bundle $(\mathcal{E}, \theta , h)$ is defined as a harmonic bundle such that on some disc $\Delta$ centered at $p$ there exists some ramified covering 
\[
    \phi_{N} \colon \Delta_{N} \to \Delta, \quad u\mapsto t = u^{N}
\]
such that $\phi_{N}^* (\mathcal{E}, \theta , h)$ is an unramified wild harmonic bundle. 

The parabolic structure on $\phi_{N}^* \mathcal{E}$ at $p\in D$ stands for the decreasing filtration by locally free subsheaves 
\[
    (\phi_{N}^* \mathcal{E})_{\alpha} = \left\{ e\in (\phi_{N}^* \mathcal{E}) (U)\,\vert \, h(e,e)^{1/2} = O \left( |u|^{- N \alpha - \epsilon } \right) \, \mbox{for all} \, \epsilon>0 \right\}. 
\]
Equivalently, it is a filtration by vector subspaces:
\begin{equation*}
    \restr{\phi_{N}^* \mathcal{E}}{p}=l_{p}^0\supset l_{p}^1\supset  l_{p}^2\supset \cdots \supset l_{p}^r = \{0\},
\end{equation*}
the jumps coming with parabolic weights:
\begin{equation*}
    0 \leq\alpha_{p}^1<\alpha_{p}^2<\cdots < \alpha_{p}^r< 1 .
\end{equation*}
% We assume, that the Higgs field is weakly parabolic, that is:
% \begin{equation*}
%     \theta:l_{p_i}^j\rightarrow l_{p_i}^j\otimes K(D),\hspace{0.3cm}\forall j\in\{0,1,2,3\}.
% \end{equation*}
The parabolic degree of $\mathcal{E}$ is defined to be
\begin{equation*}
\textrm{pdeg}_{\alpha}\mathcal{E}=\textrm{deg}\mathcal{E}+\sum_{p\in D}\sum_{j=1}^r \alpha_{p}^j.
\end{equation*}
For a proper holomorphic subbundle $\mathcal{F}\subset\mathcal{E}$ satisfying $\theta:\mathcal{F}\rightarrow\mathcal{F}\otimes K(D)$, the parabolic degree is defined to be:
\begin{equation*}
\textrm{pdeg}_{\alpha}\mathcal{F}=\textrm{deg}\mathcal{F}+\sum_{p\in D}\sum_{j=1}^r \alpha_{p}^j\textrm{dim}((\restr{\mathcal{F}}{p}\cap l_{p}^{j-1})/(\restr{\mathcal{F}}{p}\cap l_{p}^{j})).
\end{equation*}
The meromorphic Higgs bundle, together with the parabolic structure is said to be $\alpha$-stable, if for any $\mathcal{F}$:
\begin{equation*}
    \frac{\textrm{pdeg}_{\alpha}\mathcal{E}}{\textrm{rank}\mathcal{E}}>\frac{\textrm{pdeg}_{\alpha}\mathcal{F}}{\textrm{rank}\mathcal{F}}.
\end{equation*}

Let us fix a preferred $\alpha$-stable wild harmonic bundle $(\mathcal{E}, \theta , h)$ that has the desired local forms at each $p\in D$. 
We may choose a global ramified covering $\phi\colon \widetilde{X}\to X$ such that $\phi^* (\mathcal{E}, \theta , h)$ is an unramifiedly wild harmonic bundle near every $p\in D$. 
We then define weighted Sobolev spaces of endomorphisms of $\phi^* E$ with respect to local polar coordinates on $\widetilde{X}$ (rather than on $X$). 
This allows us to define an (infinite-dimensional) affine space of connections $\widetilde{\mathcal{A}}$ and gauge group $\widetilde{\mathcal{G}}$ by~\cite[Equations~(2.5)-(2.6)]{BB}. 
Let $\Gamma$ denote the group of deck transformations of $\phi$. 
Then, $\Gamma$ naturally acts on $\widetilde{\mathcal{A}}$ and on $\widetilde{\mathcal{G}}$, over the identity of $\widetilde{X}$. 
Let us set 
\[
    \mathcal{A} = \widetilde{\mathcal{A}}^{\Gamma}, \quad \mathcal{G} = \widetilde{\mathcal{G}}^{\Gamma} 
\]
the $\Gamma$-invariant sections of $\widetilde{\mathcal{A}}, \widetilde{\mathcal{G}}$. 
The gauge theory of~\cite[Section~5]{BB} carries over to this $\Gamma$-equivariant setup. 
In particular, the equivariant version of~\cite[Theorem~5.4]{BB} holds true: the moduli space of $\Gamma$-invariant solutions of the self-duality equations on $\widetilde{X}$ whose difference from $\phi^*(\mathcal{E}, \theta , h)$ satifies suitable decay conditions exists, and is a smooth hyperK\"ahler manifold. 
By conformal equivalence of the equations and Galois descent, such solutions are in bijection to solutions of the self-duality equations on $X$ whose difference from $(\mathcal{E}, \theta , h)$ satifies suitable decay conditions.

Using the same argument as~\cite[Sections~7,8]{BB}, unitary gauge equivalence classes of solutions to the self-duality equations are in bijection with complex gauge equivalence classes of $\alpha$-stable meromorphic Higgs bundles having the required normal forms, and also with complex gauge equivalence classes of $\beta$-stable irregular singular connections having the required normal forms. 

In all cases, completeness follows from~\cite[Theorem~0.2]{BB} too, provided that ${\mathcal{M}}_{Dol}^{JKT*}$ is smooth, which is the case for generic choices of the parameters (see~\cite[Section~8.1]{BB}), that we have assumed. 

The interpretation in terms of integrable connections with suitably fixed irregular types also follows from~\cite{BB}. 
The transformation of the irregular types from Dolbeault to de Rham local forms is $\frac 12 q_i \mapsto q_i$. 
The transformation of the $\nu_i$ eigenvalues of the residues and the $\alpha_i$ parabolic weights obeys Simpson's formulas:
\begin{equation}
\label{eq:simpson}
    \alpha_i = \operatorname{Re} \mu_i, \quad \nu_i = \frac{\mu_i - \beta_i}{2}, 
\end{equation}
where $\mu_i$ are the eigenvalues of the residues and $\beta_i$  the parabolic weights on the de Rham side. 

As for the dimension, it is easy to compute. 
In Section~\ref{sec:Dol} we will give a construction in terms of elliptic pencils, while the constructions of the diffeomorphic wild character varieties as affine cubic surfaces over $\mathbb{C}$ were given in our previous paper \cite{ESz2}.
\end{proof}

The equivalence between the normal forms given in~\eqref{polar1}--\eqref{polar6} and the Puiseux coefficients of the eigenvalues of the Higgs field is formulated in the following lemma.

\begin{lemma}\label{lem:forms}
The Laurent polynomials $q_i$ appearing in~\eqref{eq:decomposition} can be given by the below expansions~\eqref{eq:qi_JKTVI}--\eqref{eq:qi_JKTI} (up to terms of nonnegative degree).

    Conversely, assume that the eigenvalues have the forms \eqref{eq:qi_JKTVI}-\eqref{eq:qi_JKTI} (up to terms of non-negative degree). Then there exist polynomial gauge transformations in the indeterminate $w$, such that the polar parts of \eqref{laurent} in the six cases have the normal forms~\eqref{polar1},~\eqref{polar2},~\eqref{polar3},~\eqref{polar4},~\eqref{polar5},~\eqref{polar6}, respectively. 
    \begin{enumerate}
        \item The Laurent polynomials equivalent to~\eqref{polar1} are
        \begin{equation}\label{eq:qi_JKTVI}
    \frac{1}{2}q_i = - a_i w^{-1}, \quad 0 \leq i \leq 2, 
\end{equation}
with corresponding residue eigenvalues $b_i$. 

        \item The Laurent polynomials equivalent to~\eqref{polar2} are
        \begin{gather}
    \frac{1}{2}q_{0,1}=q_{\pm}=\lambda_2 w^{-1/2}-a_0w^{-1}\nonumber \\
    \frac{1}{2}q_2=-a_1w^{-1}. \label{eq:qi_JKTV}
\end{gather}
with corresponding residue eigenvalues $0,b_1,b_2$. Here $q_{0,1}$ is a bi-valued branched meromorphic function, depending on the values of the square root. The coefficient $\lambda_2$ depends only on $b_0,b_1$.
See~\cite[Equations~(4.5a),~(4.5b)]{JKTI}.

        \item The Laurent polynomials equivalent to~\eqref{polar3} are
        \begin{gather}
    \frac{1}{2}q_{0,1,2}= \lambda_1w^{-1/3}+\lambda_2w^{-2/3}-a_0w^{-1} \label{eq:qi_JKTIVa} 
\end{gather}
with the corresponding residue eigenvalues $0,0,b_2$. Here, $q_{0,1,2}$ is a tri-valued branched meromorphic function, depending on the values of the cubic root. The coefficients $\lambda_2,\lambda_1$ depend only on $b_0,b_1,b_2$. 

        \item The Laurent polynomials equivalent to~\eqref{polar4} are
        \begin{gather}\label{eq:qi_JKTIVb}
    \frac{1}{2}q_i = -b_iw^{-1}-\frac{1}{2}a_iw^{-2}, \quad 0 \leq i \leq 2, 
\end{gather}
with corresponding residue eigenvalues $c_i$. 

        \item The Laurent polynomials equivalent to~\eqref{polar5} are
        \begin{gather}
     \frac{1}{2}q_{0,1}=q_{\pm}=\lambda_1w^{-1/2}+\lambda_2w^{-1}+\lambda_3w^{-3/2}-\frac{1}{2}a_0w^{-2} \nonumber \\
     \frac{1}{2}q_2=-b_2w^{-1}-\frac{1}{2}a_1w^{-2}. \label{eq:qi_JKTII}
\end{gather}
with corresponding residue eigenvalues $0,c_1,c_2$. Here, $q_{0,1}$ is a bi-valued branched meromorphic function, depending on the values of the square root. The coefficients $\lambda_1,\lambda_2,\lambda_3$ depend only on $b_0,b_1,c_0,c_1$.

        \item The irregular parts equivalent to \eqref{polar6} are
        \begin{gather}\label{eq:qi_JKTI}
    \frac{1}{2}q_{0,1,2} =\lambda_1w^{-1/3}+\lambda_2w^{-2/3}+\lambda_3w^{-1}+\lambda_4w^{-4/3}+\lambda_5w^{-5/3}-\frac{1}{2}a_0w^{-2},
\end{gather}
with corresponding residue eigenvalues $0,0,c_2$. Here, $q_{0,1,2}$ is a tri-valued branched meromorphic function, depending on the values of the cubic root. The coefficients $\lambda_i$ ($1\leq i\leq 5$) depend only on $b_i,c_i$.
    \end{enumerate}
    
\end{lemma}
\begin{proof}
    The particular cases of interest to us can be calculated directly using the quadratic and cubic formulas. The statement also follows from~\cite[Theorem~1.3]{KSz}.
    See~\cite[Theorem 2.1]{KSz} for a more general statement in arbitrary rank.
\end{proof}

\section{Dolbeault spaces}\label{sec:Dol}

In this section, we adapt~\cite{ASz} along the lines of~\cite[Proposition~6.4]{ISSz1},~\cite{ISSz2} to construct, for each $* \in \{VI,V,IVa,IVb,II,I\}$, a surface $Z^*$ such that the BNR (or spectral) correspondence~\cite{BNR} establishes an algebraic isomorphism between $\mathcal{M}_{\textrm{Dol}}^{JKT*}$ and the moduli space $\mathcal{M}_{\textrm{Muk}}^{*}$ of sheaves on $Z^*$ with suitably fixed Chern classes~\cite{Muk}. 
The results of this section will be used for the proof of algebraic symplectic isomorphism of the Dolbeault spaces Theorem~\ref{thm:Dolisometry}. 
The surface $Z^*$ can be defined as a suitable blow-up of the Hirzebruch surface of index one, followed by the blow-down of its section at $\infty$:
\[
   \omega\colon Z^* \dashrightarrow  \mathbb{F}_1.
\]
As it is standard~\cite{BNR} in the case of regular Higgs bundles without poles, we will see that in each of our cases, all the spectral curves are disjoint from the section at infinity $\sigma_{\infty}$, so that its blow-down is an isomorphism restricted to the spectral curves. 
We have $[\sigma_{\infty}] \cap [\sigma_{\infty}] = -1$, where $\cap$ stands for the intersection form on $H_2 (Z^*, \mathbb{Q})$. 
%The algebraic version of Nahm transformation. 
We note that a general spectral description of such symplectic leaves has been provided by~\cite{Sz_BNR} in the unramified case and by Diaconescu--Donagi--Pantev~\cite{DDP} in general. 

Since the length of $D$ is 3, we have the isomorphism of line bundles:
\begin{equation*}
    K_{\mathbb{C}P^1}(D)\cong\mathcal{O}_{\mathbb{C}P^1}(1).
\end{equation*}
Let us consider the graded sheaf of $\mathcal{O}_{\mathbb{C}P^1}$-algebras 
\[
    \operatorname{Sym}^{\bullet} (K_{\mathbb{C}P^1}(D)^{-1}) = \bigoplus_{j=0}^{\infty} \operatorname{Sym}^j (K_{\mathbb{C}P^1}(D)^{-1})
\]
The fiber-wise projectivization~\cite[Definition~II.5.7]{MO} of the line bundle $K_{\mathbb{C}P^1}(D)$ is then defined as 
\[
    \mathbb{F}_1 = \mathbb{P} (K_{\mathbb{C}P^1}(D) \oplus \mathcal{O}_X) =  \operatorname{Proj} \left(\operatorname{Sym}^{\bullet} (K_{\mathbb{C}P^1}(D)^{-1}) \right) , 
\]
and is called the Hirzebruch surface of index one. 
This construction immediately gives a canonical fiber bundle structure 
\[
    p\colon \mathbb{F}_1\rightarrow\mathbb{C}P^1
\]
called the ruling, whose fibers are isomorphic to $\mathbb{C}P^1$. 
Said differently, $\mathbb{F}_1$ is a rational ruled surface. 
The relative ample bundle of $p$, denoted by $\mathcal{O}_{\mathbb{F}_1|\mathbb{C}P^1}(1)$, is defined by the sheaf of graded $\operatorname{Sym}^{\bullet} (K_{\mathbb{C}P^1}(D)^{-1})$-modules 
\[
    \left( \operatorname{Sym}^{\bullet} (K_{\mathbb{C}P^1}(D)^{-1}) \right) (1) = \bigoplus_{j=0}^{\infty} \operatorname{Sym}^{j+1} (K_{\mathbb{C}P^1}(D)^{-1}), 
\]
see~\cite[Section~III.1]{MO}. 
We infer 
\begin{align}
    p_* \mathcal{O}_{\mathbb{F}_1|\mathbb{C}P^1}(1) & \cong \operatorname{Sym}^{\bullet} (K_{\mathbb{C}P^1}(D)^{-1}) / \operatorname{Sym}^{\bullet} (K_{\mathbb{C}P^1}(D)^{-1})(2) \notag \\
    & \cong \operatorname{Sym}^0 (K_{\mathbb{C}P^1}(D)^{-1}) \oplus \operatorname{Sym}^1 (K_{\mathbb{C}P^1}(D)^{-1}) \notag \\
    & = \mathcal{O}_{\mathbb{C}P^1} \oplus K_{\mathbb{C}P^1}(D)^{-1}. \label{eq:direct_image_relative_ample_sheaf}
\end{align}
In particular, we get canonical sections of $\mathcal{O}_{\mathbb{F}_1|\mathbb{C}P^1}(1)$ and of $\mathcal{O}_{\mathbb{F}_1|\mathbb{C}P^1}(1) \otimes p^* K_{\mathbb{C}P^1}(D)$ by pulling back by $p$ the natural section $1 \in H^0 ( \mathbb{C}P^1 , \mathcal{O}_{\mathbb{C}P^1})$. 
Let us denote these canonical sections by 
\begin{equation*}
    \zeta\in H^0(\mathbb{F}_1,p^*K_{\mathbb{C}P^1}(D)\otimes\mathcal{O}_{\mathbb{F}_1|\mathbb{C}P^1}(1)),\hspace{0.3cm}\xi\in H^0(\mathbb{F}_1,\mathcal{O}_{\mathbb{F}_1|\mathbb{C}P^1}(1)). 
\end{equation*}
The subvariety of $\mathbb{F}_1$ defined by $\xi = 0$ is denoted by $\sigma_{\infty}$ and called the section at infinity of $p$. 
The Zariski open subset $\mathbb{F}_1 \setminus \sigma_{\infty}$ is naturally isomorphic to the total space $\operatorname{Tot}(K_{\mathbb{C}P^1}(D))$ of the line bundle $K_{\mathbb{C}P^1}(D)$ over $\mathbb{C}P^1$. 
% The restriction of $\mathcal{O}_{\mathbb{F}_1|\mathbb{C}P^1}(1)$ to $\operatorname{Tot}(K_{\mathbb{C}P^1}(D))$ is isomorphic to the sheaf of regular functions $\mathcal{O}_{\mathbb{F}_1}$. 
The restriction of $\zeta \xi^{-1}$ to $\operatorname{Tot}(K_{\mathbb{C}P^1}(D))$ agrees with the canonical section of $p^* K_{\mathbb{C}P^1}(D)$.
\begin{defn}
    The \emph{spectral curve} $\Sigma_{(\mathcal{E},\theta)}$ of $(\mathcal{E},\theta)$ is the algebraic curve determined by the equation $\chi_{\theta}(\zeta, \xi )=0$, where $\chi_{\theta}$ stands for the homogenized characteristic polynomial of $\theta$, namely:
\begin{equation}
\label{char}
    \chi_{\theta}(\zeta, \xi)=\textrm{det}(\zeta I_{\mathcal{E}}-\xi p^*\theta)=\zeta^3+F_{\theta}\zeta^2\xi+G_{\theta}\zeta\xi^2+H_{\theta}\xi^3.
\end{equation}
\end{defn}

The characteristic polynomial is a section of $p^*K_{\mathbb{C}P^1}(D)^{\otimes3}\otimes\mathcal{O}_{\mathbb{F}_1|\mathbb{C}P^1}(3)$ over $\mathbb{F}_1$, where $\mathcal{O}_{\mathbb{F}_1|\mathbb{C}P^1}(n)=\mathcal{O}_{\mathbb{F}_1|\mathbb{C}P^1}(1)^{\otimes n}$ is the notation for the tensor power. 
In particular, its characteristic coefficients $F_{\theta},G_{\theta},H_{\theta}$ satisfy: 
\begin{align*}
    F_{\theta} & \in H^0(\mathbb{C}P^1,K_{\mathbb{C}P^1}(D)^{}),\\ 
    G_{\theta} & \in H^0(\mathbb{C}P^1,K_{\mathbb{C}P^1}(D)^{\otimes2}) \\
    H_{\theta} & \in H^0(\mathbb{C}P^1,K_{\mathbb{C}P^1}(D)^{\otimes3})
\end{align*}
The direct sum of these spaces of sections is denoted by $\mathcal{B}$ and called the Hitchin base. 
Its dimension can be easily computed, as $\textrm{dim}_{\mathbb{C}}\mathcal{B}=2+3+4=9$, because of the identification 
\begin{equation*}
    H^0(\mathbb{C}P^1,K_{\mathbb{C}P^1}(D)^{\otimes n})\cong H^0(\mathbb{C}P^1,\mathcal{O}_{\mathbb{C}P^1}(n))\cong \mathbb{C}^{n+1}.
\end{equation*}
The characteristic coefficients of arbitrary rank $3$ meromorphic Higgs bundles on $\mathbb{C}P^1$ with poles bounded by $D$ (i.e., elements of the Poisson moduli space considered in~\cite{Mark}) take values in this Hitchin base. 
We are interested in the image of the symplectic leaves. 
It is known in general~\cite[Corollary~8.10]{Mark} that the symplectic leaves of the Poisson moduli space with varying numerical parameters are given by a coset-foliation on $\mathcal{B}$. 
In the cases under consideration, in concrete terms we have: 
\begin{prop}
For fixed values of the parameters in the local forms of $\theta$, the set of possible $F_{\theta},G_{\theta},H_{\theta}$ forms an affine subspace of $\mathcal{B}$ of dimension $1$. 
\end{prop}
\begin{proof}
The proof mimics those of other similar results. 
For the logarithmic rank $3$ case whose deformations are the current irregular curves, see~\cite[Section~2.3]{ESz}.
For ramified irregular singularities in the rank $2$ Painlev\'e cases, see~\cite[Proposition~6.4]{ISSz1},~\cite[Section~5.1]{ISSz2}.

Since we fix the parameters appearing in the polar part of $\theta$ in the singular points, and the length of $D$ is $3$, fixing this data decreases the degree of freedom by $3$ in each direct summand of $\mathcal{B}$ (by $2$ in the first component, where we have a redundant condition). 
The imposed conditions are linear. 
We infer that one degree of freedom remains. 
\end{proof}

\begin{defn}
The map
\begin{align*}
    h:\mathcal{M}_{Dol}^{JKT*} & \rightarrow\mathbb{C} \\ %,\hspace{0.3cm}(\mathcal{E},\theta)\mapsto\textrm{det}(\theta)\\
    (\mathcal{E}, \theta) & \mapsto(F_{\theta},G_{\theta},H_{\theta})
\end{align*}
is called the \emph{Hitchin fibration}. 
\end{defn}
Here, we have fixed an affine isomorphism between the image of $h$ and $\mathbb{C}$, as it is justified by the proposition.
For a given value $b_0 \in \mathcal{B}$, the fiber $h^{-1}(b_0)$ consists of pairs $(\mathcal{E},\theta)$ for which $\textrm{det}(\theta) = b_0$. 

\begin{theorem}[Beauville--Narasimhan--Ramanan correspondence \cite{BNR},{\cite[Theorem 5.4]{Sz_BNR}}]\label{thm:BNR}
    There is an equivalence of categories between the groupoid of irregular Higgs bundles of rank $3$ and degree $d$ with prescribed polar part and the relative Picard scheme parameterizing torsion-free sheaves of rank $1$ and of degree $d+3$ on holomorphic curves in a given homology class $[F_{\infty}^{*}]\in H_2 (Z^*, \mathbb{Z})$ on a certain multiple blow-up $Z^*$ of the Hirzebruch surface $\mathbb{F}_1$. 
\end{theorem}

As a matter of fact, we will see that $\mathcal{O}_{Z^*}(F_{\infty}^{*}) = K_{Z^*}^{-1}$, in other words $(Z^*, F_{\infty}^{*})$ is a log-Calabi--Yau surface. 
We will provide details of the construction of the blow-ups $Z^*$ and the curves $F_{\infty}^{*}$ in the cases under investigation in Sections~\ref{sec:DolVI}--\ref{sec:DolI}. 
The proof of the theorem will be given in Section~\ref{sec:spectral_sheaf}. 

\begin{comment}
There is also an alternative reformulation of the BNR-correspondence in \cite{SWW}, regarding parabolic Higgs bundles. It states that there is a one-to-one correspondence between parabolic Higgs bundles of a fixed degree with prescribed generic characteristic coefficients and line bundles over the corresponding normalized spectral curve with a fixed degree (established by the parabolic data). 
As a consequence of this theorem, we can conclude the following proposition regarding the exact case we are investigating.
\end{comment}

As a consequence of this theorem, we can conclude the following proposition regarding the exact case we are investigating.

\begin{prop}
    For generic $b_0 \in \mathcal{B}$, the fiber $h^{-1}(b_0 )$ of the Hitchin fibration is a torsor over the Jacobian $\operatorname{Jac}(\Sigma_{(\mathcal{E},\theta)})$ for any $(\mathcal{E},\theta)\in h^{-1}(b_0 )$.
    In particular, $h^{-1}(b_0 )$ is a smooth elliptic curve. 
\end{prop}

\begin{proof}
Fixing $b_0$ amounts to fixing the support $\Sigma_{b_0}$ of $\mathcal{S}_{(\mathcal{E},\theta)}$. 
For generic choice of $b_0 \in \mathcal{B}$, the curve $\Sigma_{b_0}$ is smooth. 
The set of isomorphism classes of torsion free rank $1$ sheaves $\mathcal{S}$ of given degree on $\Sigma_{b_0}$ is then a torsor over the Jacobian $\operatorname{Jac}(\Sigma_{b_0})$. 

% In case $\Sigma_{(\mathcal{E},\theta)}$ is irreducible, it is a smooth cubic curve; therefore, Jac$(\Sigma)$ is compact. On the other hand, if $\Sigma_{(\mathcal{E},\theta)}$ is reducible, then Jac$(\Sigma_{(\mathcal{E},\theta)})$ is non-compact. In order to provide suitable compactification, stability condition of the Higgs bundles is needed, which comes from the given parabolic structure. 

% This way, we see a fiber of the Hitchin fibration is defined by the spectral curve together with a torsor over its Jacobian. In our case, this is of complex dimension 1; therefore, the generic fiber of $h$ turns out to be a smooth elliptic curve and $h$ to be an elliptic fibration.  
\end{proof}

From now on, fix some $b_0 \in \mathcal{B}$, and let us use the notation $\Sigma=\Sigma_{b_0}$. Consider the section
\[
s_1:=\chi(\zeta,\xi)\in H^0(\mathbb{F}_1,p^*K_{\mathbb{C}P^1}(D)^{\otimes 3}\otimes\mathcal{O}_{\mathbb{F}_1|\mathbb{C}P^1}(3)),
\]
and define the divisor of zeros $C_1:=(s_1)$, i.e. $C_1$ is the spectral curve $\Sigma$ as a Weil divisor. Define also the Weil divisor $C_2$, as $C_2:=3F+ 3 \sigma_{\infty}$ (in cases $*\in\{I,II,IVb\}$) or $C_2=F_1+2F_2+ 3 \sigma_{\infty}$ (in cases $*\in\{IVa,V,VI\}$). 
Here $F,F_1,F_2$ $(F_1\neq F_2)$ are generic fibers of the ruling $p:\mathbb{F}_1\rightarrow\mathbb{C}P^1$.
Specifically, because of our conventions, $F=F_2$ is the fiber of $p$ in $\mathbb{F}_1$ over $\infty \in \mathbb{C}P^1$, and $F_1$ is the fiber of $p$ in $\mathbb{F}_1$ over $0 \in \mathbb{C}P^1$. 

\begin{lemma}
    $C_1$ and $C_2$ are linearly equivalent. 
\end{lemma}

\begin{proof}

We have $[C_1]=[C_2]\in H_2(\mathbb{F}_1,\mathbb{Z})$, because of the linear relation
\[
3[\sigma_{\infty}]+3[F]=3[\sigma]
\]
for some arbitrary $\sigma$ section of $p$, and obviously $[F_1]=[F_2]=[F]$. Since $C_1$ is a triple section, $C_1$ and $C_2$ are homologous; therefore, they are linearly equivalent. (It is proved in \cite[Chapter 4.3]{GH} that on a rational ruled surface, such as $\mathbb{F}_1$, two curves are linearly equivalent if and only if they are homologous.)
\end{proof}

As a consequence of this lemma, we have $C_2=(s_2)$ for some section 
\[
    s_2 \in H^0(\mathbb{F}_1,p^*K_{\mathbb{C}P^1}(D)^{\otimes 3}\otimes\mathcal{O}_{\mathbb{F}_1|\mathbb{C}P^1}(3)).
\]
The set  
\begin{equation}
    \label{pencil}
        \{(\lambda_1s_1+\lambda_2s_2)|[\lambda_1:\lambda_2]\in\mathbb{C}P^1\}
    \end{equation}
is a complete family of linearly equivalent divisors (i.e. a pencil), generated by $C_1$ and $C_2$.

\begin{defn}
    $Z^*$ is constructed as the sequence of 9 blow-ups (with possibly infinitesimally close ones) at the base points $C_1\cap C_2$ of \eqref{pencil}, and the blow-down of $\sigma_{\infty}$. See Sections~\ref{sec:DolVI}--\ref{sec:DolI} for the details in each case.
\end{defn}

The blow-down of $\sigma_{\infty}$ is included in order to get a minimal fibration. 
Denote the total transforms of $\Sigma$ and $F,F_1,F_2$ after this process by $\widetilde{\Sigma},\widetilde{F},\widetilde{F}_1,\widetilde{F}_2$, respectively.

\begin{defn}
    Let $F_{\infty}^{*}$ be the total transform of $C_2$ after the sequence of the blow-ups and the blow-down. In concrete terms, 
    \begin{equation}\label{eq:finf}
        F_{\infty}^{*}=3\widetilde{F}+\sum_{i=1}^k E_i \hspace{0.2cm}\left(\textrm{or}\hspace{0.2cm}\widetilde{F}_1+2\widetilde{F}_2+\sum_{i=1}^k E_i\right)
    \end{equation}
    where $E_i$'s are the exceptional divisors appearing after the blow-ups, and $k$ depends on the particular $*$-case (see Sections~\ref{sec:DolVI}--\ref{sec:DolI}).
\end{defn}

\begin{lemma}\label{lemma:selfint}
    $[F_{\infty}^{*}]\cap[F_{\infty}^{*}]=0$ in $H_2(Z^*,\mathbb{Z})$.
\end{lemma}
\begin{proof}
    \begin{align*}
        [C_2]\cap[C_2] & =(3[F]+3[\sigma_{\infty}])\cap(3[F]+3[\sigma_{\infty}]) \\
        & = 9([F]\cap[F]+2[F]\cap[\sigma_{\infty}]+[\sigma_{\infty}]\cap[\sigma_{\infty}]) \\
        & = 9(0+2-1)=9.
    \end{align*}
    Each and every one of the nine blow-ups reduces this self-intersection number by one, whereas the blow-down of $\sigma_{\infty}$ does not change it. In the end, the total transform of $C_2$ has self-intersection 0.
\end{proof}

\begin{prop}
    There exist elliptic fibrations $Z^*\rightarrow \mathbb{C}P^1$ with fiber $F_{\infty}^{*}$ over $\infty\in\mathbb{C}P^1$, whose complement is biregular to the moduli space $\mathcal{M}_{Dol}^{JKT*}$.
\end{prop}

\begin{remark}
    The fibration given by the proposition is not to be confused with the ruling $p$. 
    It gives the Hitchin fibration, up to applying the relative Abel--Jacobi map. 
\end{remark}

\begin{proof}
Consider the set $|F_{\infty}^{*}|$ of linearly equivalent effective divisors (or linear system) generated by $F_{\infty}^{*}\subset Z^*$. 
An arbitrary projective curve $C\subset Z^*$ belongs to $|F_{\infty}^{*}|$ if and only if $C$ satisfies the following system of equations.

\begin{equation}\label{eq:fibration_equations}
    \begin{cases}
      [C]\cap[\widetilde{F}]=0, \\
      [C]\cap[E_i]=0,\hspace{0.2cm}1\leq i\leq k \\
      [C]\cap[s_i]=1, \hspace{0.2cm}1\leq i\leq l
    \end{cases}
\end{equation}
where $E_i,s_i$ are the exceptional divisors and the sections coming from the blow-ups. 
By sections $s_i$, we mean the exceptional divisors corresponding to the leaves of the resolution tree; they are indeed sections of the fibration.  
The remaining exceptional divisors are labelled by $E_i$.
Their numbers $k,l$ depend on the particular cases, but we always have $k+l=9$ (see Sections~\ref{sec:DolVI}--\ref{sec:DolI}).

One can easily check that $C=F_{\infty}^{*}$ satisfies this system of equations. 
On the other hand $H_2(Z^*,\mathbb{Z})\cong\mathbb{Z}^{10}$, and is endowed with a  non-degenerate intersection lattice, generated by $E_1,...,E_k,s_1,...,s_l,\widetilde{F}$. 
Therefore, the above 10 equations uniquely determine the homology class $[F_{\infty}^{*}]$ in $H_2(Z^*,\mathbb{Z})$. Now $C$ satisfies the system of equations if and only if $[C]=[F_{\infty}^{*}]$, which means that $C$ and $F_{\infty}^*$ are linearly equivalent. 
(Recall again that two curves on the rational surface $Z^*$ are linearly equivalent if and only if they are homologous, as it follows from \cite[Chapter 4.3]{GH}.)

\begin{lemma}\label{lem:pencil}
We have 
\begin{equation*}
H^0(Z^*,\mathcal{O}_{Z^*}(F_{\infty}^{*}))\cong\mathbb{C}^2,
\end{equation*}
thus $\operatorname{dim}|F_{\infty}^{*}|=1$, i.e. $|F_{\infty}^{*}|$ is a pencil. 
\end{lemma}
\begin{proof}
Consider the following short exact sequence of sheaves on $Z^*$:
\begin{equation*}
    0 \to \mathcal{O}_{Z^*} \to \mathcal{O}_{Z^*}(F_{\infty}^{*}) \to \mathcal{O}_{F_{\infty}^{*}}(F_{\infty}^{*}) \to 0.
\end{equation*} 
The corresponding long exact sequence in cohomology is
\begin{equation*}
    0 \to H^0(Z^*,\mathcal{O}_{Z^*}) \to H^0(Z^*,\mathcal{O}_{Z^*}(F_{\infty}^{*})) \to H^0(Z^*,\mathcal{O}_{F_{\infty}^{*}}(F_{\infty}^{*})) \to H^1(Z^*,\mathcal{O}_{Z^*})= 0.
\end{equation*} 
Because of $[F_{\infty}^{*}]\cap[F_{\infty}^{*}]=0$ (see Lemma~\ref{lemma:selfint}), we have 
\[
H^0(Z^*,\mathcal{O}_{F_{\infty}^{*}}(F_{\infty}^{*}))=H^0(Z^*,\mathcal{O}_{F_{\infty}^{*}})=H^0(F_{\infty}^*,\mathcal{O}_{F_{\infty}^{*}})\cong\mathbb{C}.
\]
The desired isomorphism now follows from the cohomology long exact sequence.
\end{proof}

One can also check that $C=\widetilde{\Sigma}$ satisfies the above system of equations, hence $[\widetilde{\Sigma}]=[F_{\infty}^{*}]\in H_2(Z^*,\mathbb{Z})$, and $\widetilde{\Sigma}\in|F_{\infty}^{*}|$. 
That is, the total transforms of $C_1$ and $C_2$ are in the same pencil. The self-intersection of the elements of this pencil is zero, therefore $Z^*\rightarrow |F_{\infty}^{*}|\cong\mathbb{C}P^1$ is an elliptic fibration. Moreover
\[
h':Z^*\setminus F_{\infty}^*\rightarrow\mathbb{C}P^1\setminus\infty
\]
is indeed Abel--Jacobi dual to the Hitchin fibration.

\end{proof}

In the sequel, we visualize the pencil arrangement \eqref{pencil}, and the resulting picture with $F_{\infty}^{*}$ in each case. In addition, we will see the geometric significance of the parameters appearing in the local forms \eqref{polar1}-\eqref{polar6}. 
For this purpose, as mentioned earlier, we always assume that the irregular singularity is located at $\infty\in\mathbb{C}P^1$.

\begin{rmrk}
    On Figures \ref{fig:fig4}-\ref{fig:fig3} we use the following notations. The curve $C_1$ is denoted by blue and $C_2$ by red. Their intersections are the basepoints of the pencils, denoted by black dots. The grey curves are exceptional divisors that do not appear in any of the pencils; they are sections. The parenthetical numbers next to the curves are the multiplicities and are considered to be 1 wherever missing. The other numbers are the self-intersection number (labeled only on the red curves). For a curve $X$, its n-th proper transform is denoted by $X^{(n)}$.
\end{rmrk}

\subsection{The logarithmic case}

For the sake of completeness, we mention this case as well, although it makes no particular interest in this work. The Dolbeault space was described in \cite{ESz}, \cite{Mih}. The singular curve at infinity of the Hitchin fibration is of type $IV$, and unlike the irregular cases, this one does not connect to any of the Painlev\'e systems. 

\subsection{Construction of $Z^{VI}$}\label{sec:DolVI}

Here we have $C_1$ a triple section of $p$ and $C_2=F_1+2F_2+3\sigma_{\infty}$. 
(For simplicity, on Figure \ref{fig:fig4} we represent the case when $C_1=L_1+L_2+L_3$ is the union of three sections of $p$ but the argument carries over verbatim to the general case.) 
There are six base points, and in all of them the curves meet transversely. We apply three blow-ups on $F_1$, and two infinitely close blow-ups at each points $a_0,a_1,a_2$ on $F_2$. Finally, we blow-down the $(-1)$-curve $\sigma_{\infty}$. It is easy to see that the fiber $F_{\infty}^{VI}$ at infinity comes from the red curve and is of type $I_0^*$. Indeed, after the process, each component of the red curve is a $(-2)$-curve (which is true in all cases). See Figure \ref{fig:fig4}.

\begin{figure}[ht]
\centering
\includegraphics[width=14.0cm]{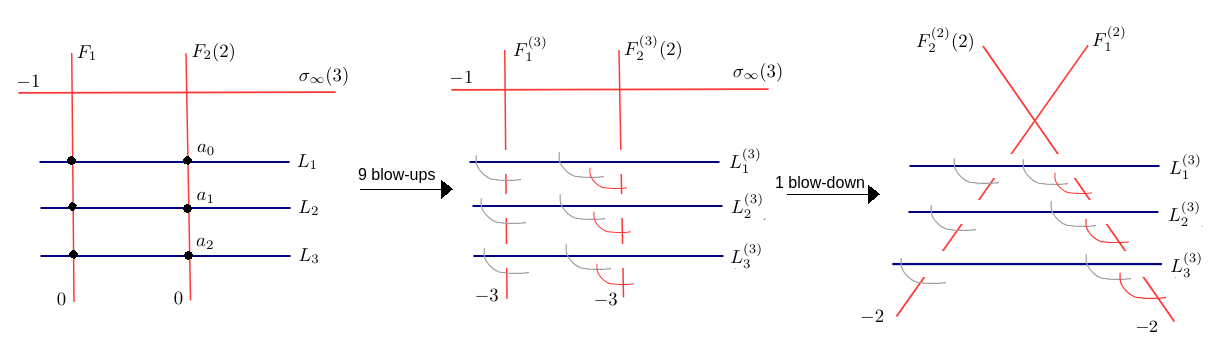}
\caption{Pencil arrangement for $JKTVI$ case.}
\label{fig:fig4}
\end{figure}

\subsection{Construction of $Z^{V}$}\label{sec:DolV}

In this case $C_1$ is again a triple section of $p$, tangent to $F_2$ at $a_0$ to second order. (On Figure \ref{fig:fig5} we represent the case when $C_1=Q+L_1$ is the union of a double section $Q$ of $p$ and a section $L_1$ of $p$, again for simplicity. 
As the notation suggests, the proper trasform of $Q$ in $\mathbb{C}P^2$ is a quadric.) 
$C_2=F_1+2F_2+3\sigma_{\infty}$, and we have 5 base points. Apply three blow-ups at the points of $F_1\cap(Q+L_1)$, while two infinitely close blow-ups at $a_1=F_2\cap L_1$ (in these base points, the curves meet transversely). At $a_0$, we need to apply four infinitely close blow-ups, until we get a section. Finally, we blow-down the $(-1)$-curve $\sigma_{\infty}$. 
The fiber $F_{\infty}^{V}$ comes from the red curve and is of type $I_1^*$. 
See Figure \ref{fig:fig5}.  

\begin{figure}[ht]
\centering
\includegraphics[width=14.0cm]{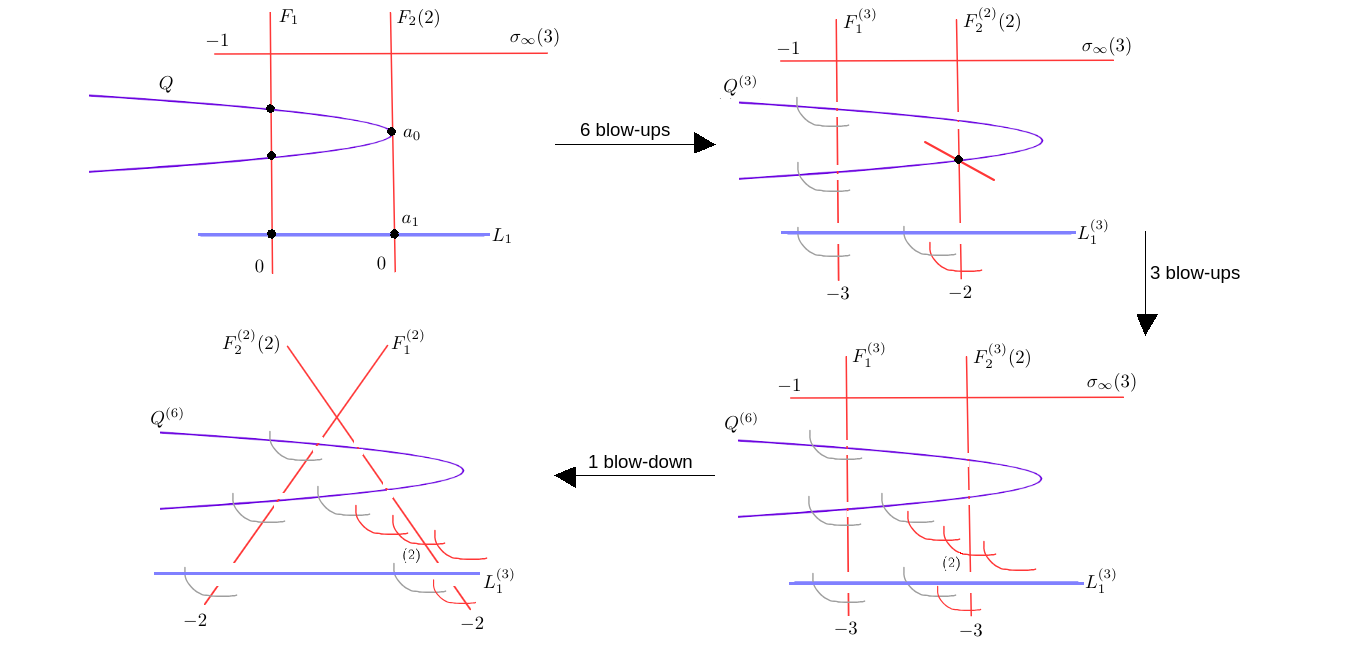}
\caption{Pencil arrangement for $JKTV$ case.}
\label{fig:fig5}
\end{figure}

\subsection{Construction of $Z^{IVa}$}\label{sec:DolIVa}

Now, $C_1=C$ is a triple section of $p$, tangent to $F_2$ at $a_0$ to third order. $C_2=F_1+2F_2+3\sigma_{\infty}$, and we have 4 base points. We need three blow-ups at the points of $F_1\cap C$, where these two meet transversely. At $a_0$, after three infinitesimally close blow-ups, we have three more exceptional divisors joining to $C_2$, and we have a new base point, where the exceptional divisor with multiplicity 3, meets $C$ (top right in Figure \ref{fig:fig6}). Here after three more blow-ups, we get a section, and finally we blow-down $\sigma_{\infty}$. After the final step, one can see that the fiber at infinity (coming from the red curve) $F_{\infty}^{IVa}$ is of type $\widetilde{E}_6$ (bottom left in Figure \ref{fig:fig6}).

\begin{figure}[ht]
\centering
\includegraphics[width=12.0cm]{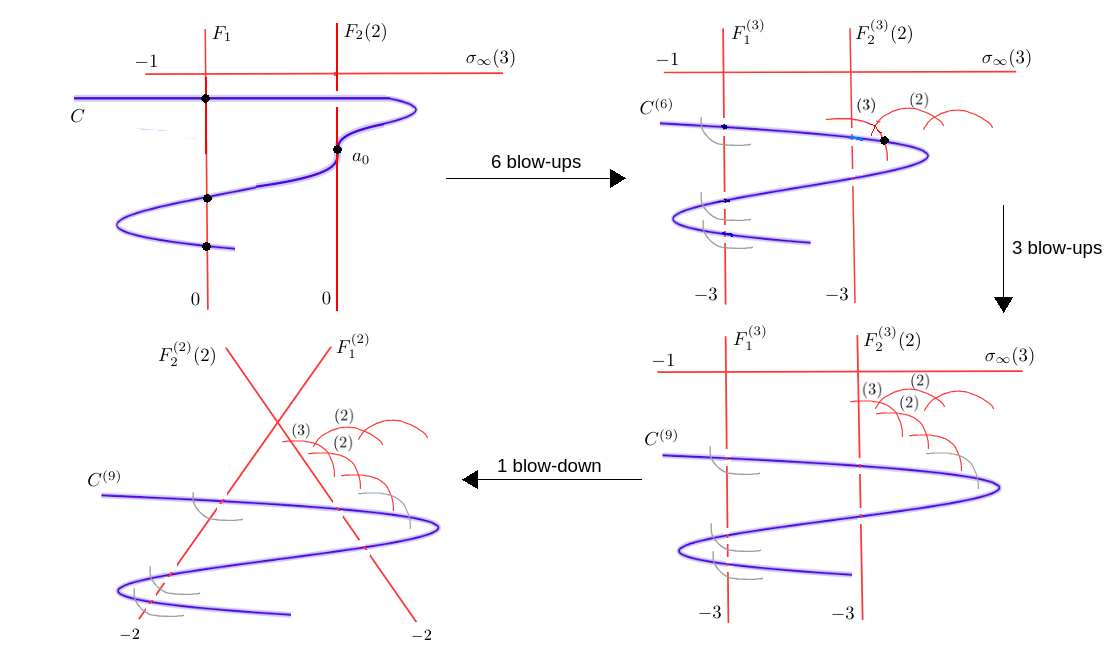}
\caption{Pencil arrangement for $JKTIVa$ case.}
\label{fig:fig6}
\end{figure}

\subsection{Construction of $Z^{IVb}$}

In this case $C_1$ is as in Section~\ref{sec:DolVI}, while $C_2=3F+3\sigma_{\infty}$, where the $F$ fiber of the ruling is counted with multiplicity three. The points $F\cap L_i=a_{i-1}$ ($i=1,2,3$) are the base points, where the curves meet transversely. We apply three infinitesimally close blow-ups in each of them. The result is visualized in Figure \ref{fig:fig1}, and if we blow-down the $(-1)$-curve $\sigma_{\infty}$ after all that, the red curve becomes the $F_{\infty}^{IVb}$ fiber at infinity. All its components are $(-2)$-curves, that is $F_{\infty}^{IVb}$ is of type $\widetilde{E}_6$.  

\begin{figure}[ht]
\centering
\includegraphics[width=12.0cm]{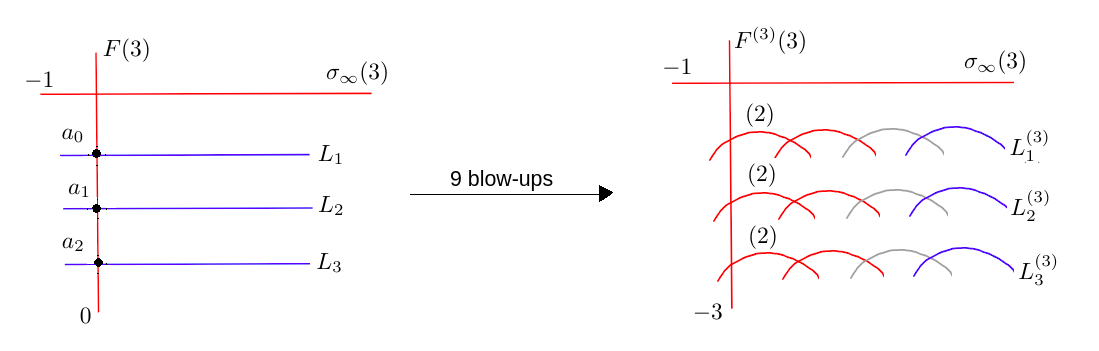}
\caption{Pencil arrangement for $JKTIVb$ case.}
\label{fig:fig1}
\end{figure}

\subsection{Construction of $Z^{II}$}

Here we have $C_1$ as in Section~\ref{sec:DolV}, and $C_2=3F+3\sigma_{\infty}$. At $a_1=F\cap L_1$, where the two curves are transverse, we apply three blow-ups and get a section. At $a_0$, the curves are doubly tangent to each other. Here after two blow-ups, we receive two exceptional divisors joining to $C_2$, and a new base point (bottom right in Figure \ref{fig:fig2}). After four more infinitesimally close blow-ups at this base point, we obtain another section, and after the blow-down of $\sigma_{\infty}$, the red curve turns to one of type $\widetilde{E}_7$, this is $F_{\infty}^{II}$.

\begin{figure}[ht]
\centering
\includegraphics[width=12.0cm]{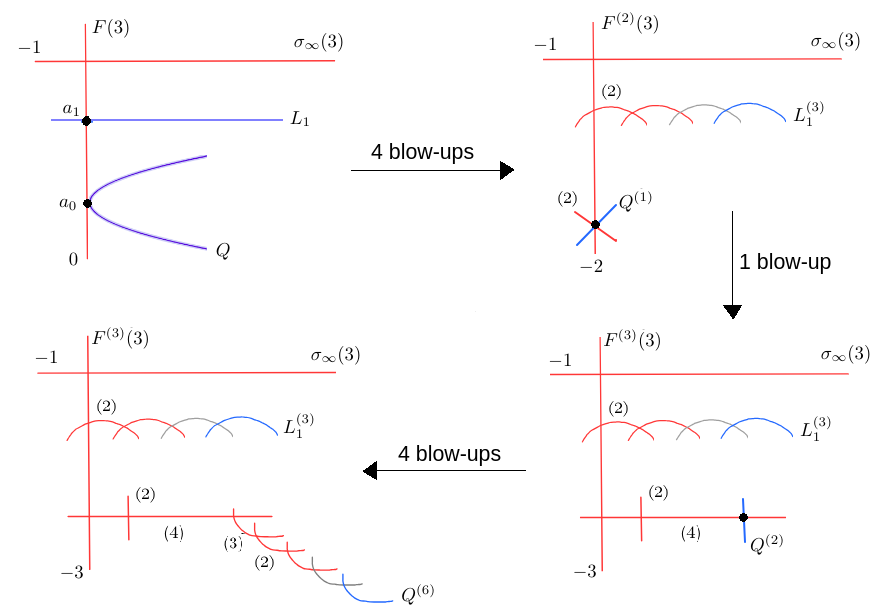}
\caption{Pencil arrangement for $JKTII$ case.}
\label{fig:fig2}
\end{figure}

\subsection{Construction of $Z^{I}$}\label{sec:DolI}

Now $C_1$ is as in Section~\ref{sec:DolV}, and $C_2=3F+3\sigma_{\infty}$. Their only intersection point is $a_0$, where they are triply tangent to each other. Apply here three infinitesimally close blow-ups, and we get three exceptional divisors joining to $C_2$, with multiplicities $2,4$ and $6$ (bottom right in Figure \ref{fig:fig3}). After that we have a new base point, where $C^{(3)}$ intersects the third exceptional divisor. We need 6 more blow-ups here to get a section. Finally, after the blow-down of $\sigma_{\infty}$, the fiber at infinity $F_{\infty}^{I}$ turns out to be of type $\widetilde{E}_8$. 

\begin{figure}[ht]
\centering
\includegraphics[width=12.0cm]{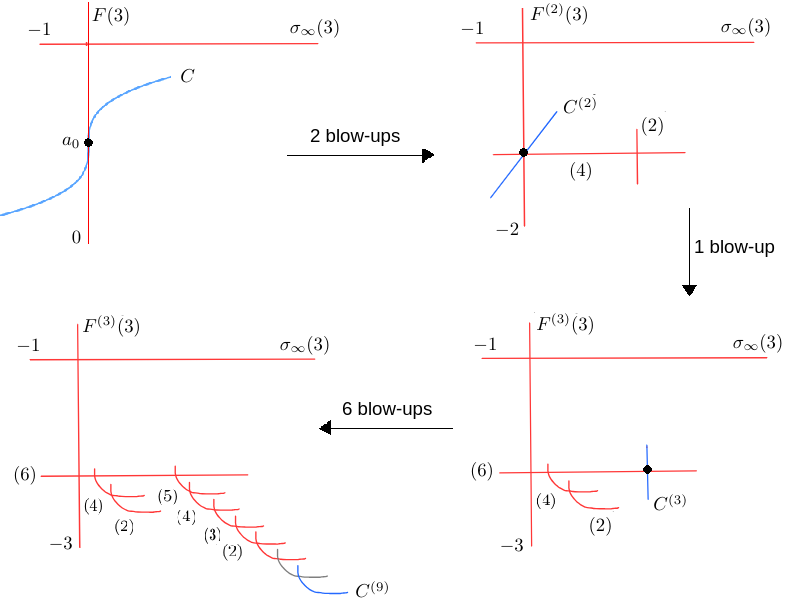}
\caption{Pencil arrangement for $JKTI$ case.}
\label{fig:fig3}
\end{figure}

\begin{prop}\label{prop:finf}
    We have $[F_{\infty}^{*}] = - [K_{Z^*}]$.
\end{prop}

\begin{proof}
    If we immediately blow down $\sigma_{\infty}$ on the original arrangement (e.g. left of Figure \ref{fig:fig4}), then the resulting surface is just $\mathbb{C}P^2$, so the red curve is the total transform of the anticanonical Weil divisor $[F_1] + 2 [F_2]=3[F]$ on $\mathbb{C}P^2$. This can be seen from $K_{\mathbb{C}P^2}\cong \mathcal{O}(-3)$, in other words, because of $K_{\mathbb{C}P^2}=-3H$, where $H$ is the hyperplane bundle on $\mathbb{C}P^2$.
\end{proof}

\begin{center}
\begin{tabular}{ |c|c|>{\centering\arraybackslash}p{3.5cm}|>{\centering\arraybackslash}p{3cm}|  }
 \hline
 \multicolumn{4}{|c|}{Summary} \\
 \hline
 $JKT$ system& polar divisor of $\theta$ &type of irregular singularity at $\infty$ & Dynkin diagram of $F_{\infty}^{*}$\\
 \hline
 \hline
 $JKTVI$   & $D=\{0\}+2\{\infty\}$    &untwisted&$I_0^*$ (or $\widetilde{D}_4$)\\
 \hline
 $JKTV$ &   $D=\{0\}+2\{\infty\}$  & minimally twisted & $I_1^*$ (or $\widetilde{D}_5$) \\
 \hline
 $JKTIVa$ &$D=\{0\}+2\{\infty\}$ & maximally twisted&$\widetilde{E}_6$\\
 \hline
 $JKTIVb$    &$D=3\{\infty\}$ & untwisted&$\widetilde{E}_6$\\
 \hline
 $JKTII$&   $D=3\{\infty\}$  & minimally twisted&$\widetilde{E}_7$\\
 \hline
 $JKTI$& $D=3\{\infty\}$  & maximally twisted& $\widetilde{E}_8$ \\
 \hline
\end{tabular}
\captionof{table}{Classification of the six investigated cases}
\label{tab1}
\end{center}

\section{Algebraic Nahm transformation}\label{sec:Nahm}

In this section we will prove Theorem~\ref{thm:Dolisometry}. 

\subsection{Spectral sheaf}\label{sec:spectral_sheaf}

\begin{defn}
  The \emph{spectral sheaf} is the coherent sheaf $\mathcal{S}_{(\mathcal{E},\theta)}$ of $\mathcal{O}_{\mathbb{F}_1}$-modules defined by the short exact sequence 
\begin{equation}\label{eq:spectral_sheaf}
    0 \to p^* ( \mathcal{E} \otimes K_X(D)^{-1}) \xrightarrow{\zeta I_{\mathcal{E}}-\xi p^*\theta}  p^* \mathcal{E} \otimes \mathcal{O}_{\mathbb{F}_1|X}(1) \to \mathcal{S}_{(\mathcal{E},\theta)} \to 0.
\end{equation}  
\end{defn}

\begin{remark}
The support of $\mathcal{S}_{(\mathcal{E},\theta)}$ is $\Sigma_{(\mathcal{E},\theta)}$. %, and as such, it is disjoint from $\sigma_{\infty}$. 
Using the birational map $Z^* \to \mathbb{F}_1$ that is an isomorphism between any smooth spectral curve $\Sigma_{(\mathcal{E},\theta)}$ and its proper transform, we may and will, without any notational distinction, consider $\mathcal{S}_{(\mathcal{E},\theta)}$ as a coherent sheaf of $\mathcal{O}_{Z^*}$-modules. 
\end{remark}
\begin{proof}[Proof of Theorem~\ref{thm:BNR}]
The proof is a modification of~\cite{Sz_BNR}. 
The isomorphism in one direction is given by 
\[
    (\mathcal{E}, \theta) \mapsto \mathcal{S}_{(\mathcal{E},\theta)}, 
\]
and the one in the other direction is 
\[
    \mathcal{S}\mapsto (p_* \mathcal{S}, p_* (\zeta \cdot) ), 
\]
where 
\[
    \zeta \cdot\colon  \mathcal{S}\to \mathcal{S}\otimes K_{\mathbb{C}P^1}(D) 
\]
stands for the multiplication operator by $\xi\in H^0 (\mathbb{F}_1, K_{\mathbb{C}P^1}(D) )$. 
Let us check that these morphisms are indeed quasi-inverses of each other. 
For this purpose, let us apply $R^{\bullet } p_*$ to~\eqref{eq:spectral_sheaf} and make use of the projection formula and the identity~\eqref{eq:direct_image_relative_ample_sheaf}. 
This gives 
\[
     0 \to \mathcal{E} \otimes K_X(D)^{-1} \to \mathcal{E} \otimes (\mathcal{O}_{X} \oplus K_X(D)^{-1}) \to R^0 p_* \mathcal{S}_{(\mathcal{E},\theta)} \to 0, 
\]
where the first map is induced by the inclusion $K_X(D)^{-1} \hookrightarrow \mathcal{O}_{X} \oplus K_X(D)^{-1}$. 
We get 
\[
    R^0 p_* \mathcal{S}_{(\mathcal{E},\theta)} \cong \mathcal{E}. 
\]
It immediately follows from~\eqref{eq:spectral_sheaf} that multiplication by $\zeta$ on the cokernel $\mathcal{S}_{(\mathcal{E},\theta)}$ agrees with the action of $\theta$ on $\mathcal{E}$. 

The only thing left to check is that $\theta$ has the required local forms~\eqref{polar1}--\eqref{polar6} if and only if the support $\Sigma_{(\mathcal{E},\theta)}$ of $\mathcal{S}_{(\mathcal{E},\theta)}$ contains scheme-theoretically the base locus of the corresponding elliptic pencil defined in Sections~\ref{sec:DolVI}--\ref{sec:DolI}, because this latter property ensures that the conditions~\eqref{eq:fibration_equations} hold for (the proper transform of) $\Sigma_{(\mathcal{E},\theta)}$. 
In the cases when the leading order term of the local form of $\theta$ at the irregular singularity $\infty\in X$ has Jordan blocks of dimension at most $2$, this goes exactly as in~\cite{ISSz1},~\cite{ISSz2}. 
%the ramification degree of $\Sigma_{(\mathcal{E},\theta)}$ at the irregular singularity $\infty\in X$ is less than or equal to $2$, 
Assume that the leading order term of the local form of $\theta$ has a single Jordan block, of dimension $3$. 
Then, it is easy to check that the ramification degree of $\Sigma_{(\mathcal{E},\theta)}$ at $\infty$ is equal to $3$, whenever $\Sigma_{(\mathcal{E},\theta)}$ is smooth. 
It is classical~\cite{Hirz} that then the resolution of the base locus of the corresponding pencil is exactly the one used in Sections~\ref{sec:DolIVa},~\ref{sec:DolI}. 
\end{proof}

\subsection{Definition and involutivity of Nahm transformation}

\begin{prop}
For each $*\in \{VI,V,IVa,IVb,II,I\}$, there exists a canonical birational map, defined on the complement of the blow-down of $\sigma_{\infty}$, 
\[
    Z^* \dashrightarrow X \times \widehat{X} ,
\]
where $X, \widehat{X}$ are projective lines over $\mathbb{C}$. 
In particular, the natural projection morphisms 
\[
    p \colon  Z^* \dashrightarrow X, \qquad \hat{p}\colon  Z^* \dashrightarrow \widehat{X}
\]
are well-defined away from the blow-down of $\sigma_{\infty}$. 
\end{prop}

\begin{proof}
Let us denote by $\operatorname{Bl}_Q (\mathbb{F}_1)$ the blow-up of $\mathbb{F}_1$ at one of the points $Q$ involved in the definition of $Z^*$. 
Remember that we have $Q\notin \sigma_{\infty}$. 
If we blow down the proper transform of the fiber of $\mathbb{F}_1$ containing $Q$, then the resulting surface is $X \times \widehat{X}$.
This gives a canonical birational map from the blow-up of $Q$ in $\mathbb{F}_1$, to $X \times \widehat{X}$.
For the construction of $Z^*$, in addition to this first blow-up of $Q$, we need to apply several other blow-ups and a single blow-down. 
Now, after the other blow-ups we still have a dominant birational morphism defined everywhere. 
On the other hand, the only blowing down involved in the definition of $Z^*$ is precisely that of $\sigma_{\infty}$. 
So, the canonical birational map is defined away from the point of $Z^*$ obtained as the blow-down of $\sigma_{\infty}$. 
\end{proof}

The just proved property, combined with the fact that $\Sigma_{(\mathcal{E}, \theta )} \cap \sigma_{\infty} = \varnothing$ for every $(\mathcal{E}, \theta ) \in \mathcal{M}_{\textrm{Dol}}^{JKT*}$, allows us to set the following: 

\begin{defn}
    The algebraic Nahm transform of $(\mathcal{E}, \theta ) \in \mathcal{M}_{\textrm{Dol}}^{JKT*}$ is 
    \[
    \mathcal{N} (\mathcal{E}, \theta ) = \left( \hat{p}_* \mathcal{S}_{(\mathcal{E}, \theta )}, -\frac 12 \hat{p}_* (p^* z\cdot ) \operatorname{d}\! \hat{z} \right).
    \]
    Here, $z\cdot$ stands for the multiplication of functions on $X$ by $z$. 
\end{defn}

\begin{prop}\label{prop:algebraicDolbeault}
Algebraic Nahm transformation is an algebraic isomorphism
\[
    \mathcal{N}\colon (\mathcal{M}_{\textrm{Dol}}^{JKT*} , \mathbf{\Omega}_{\textrm{Dol}}^{JKT*}) \to (\mathcal{M}_{\textrm{Dol}}^{P*}, \mathbf{\Omega}_{\textrm{Dol}}^{P*}). 
\]
It is an algebraic involution up to sign:
\[
    \mathcal{N}^2 = (-1)^*, 
\]
where $(-1)\colon X \to X$ is the map $z\mapsto -z$. 
\end{prop}
\begin{proof}
Algebraicity is straightforward from the definition. 
The fact that each $\mathcal{M}_{\textrm{Dol}}^{JKT*}$ gets mapped to $\mathcal{M}_{\textrm{Dol}}^{P*}$ follows from Proposition~\ref{prop:FLN} and~\cite{ESz3} (and can be checked directly from the equations of the spectral curves, too). 
Involutivity up to sign follows from the identity 
\[
    (-1)^* \mathcal{S}_{(\mathcal{E}, \theta )} = \mathcal{S}_{(\widehat{\mathcal{E}}, \widehat{\theta} )}, 
\]
exactly as in~\cite[Theorem~10.2]{ASz}.
\end{proof}

\subsection{Relationship with Fourier--Laplace transformation}

For details and references on Fourier--Laplace transformation of $\mathcal{D}$-modules, see~\cite{ESz3}. 
\begin{prop}\label{prop:FLN}
There exists a commutative diagram of diffeomorphisms between smooth $4$-manifolds: 
\begin{equation*}
    \xymatrix{\mathcal{M}_{\textrm{Dol}}^{JKT*} \ar[r]^{\mathcal{N}} \ar[d]_{NAH} & \mathcal{M}_{\textrm{Dol}}^{P*} \ar[d]^{NAH} \\
    \mathcal{M}_{\textrm{dR}}^{JKT*} \ar[r]^{\mathcal{F}} & \mathcal{M}_{\textrm{dR}}^{P*}}
\end{equation*}
Here, $NAH$ is the non-abelian Hodge diffeomorphism~\eqref{eq:NAH}, $\mathcal{F}$ is Fourier--Laplace transformation of $\mathcal{D}$-modules, and $\mathcal{N}$ is algebraic Nahm transformation of irregular Higgs bundles. 
\end{prop}

\begin{proof}
This follows the exact same argument as in~\cite[Theorem~8.5]{ASz}, see also~\cite{Sz_Nahm}. 
%TODO: L2-analysis, admissibility conditions on par. wts. 
We let 
\[
    D_{\hat{z}} = D - \frac 12 \hat{z} \operatorname{d} \! z  - \frac 12 \overline{\hat{z} \operatorname{d} \! z}
\]
define the $\hat{z}$-twisted smooth integrable connection over $X\setminus D$. 
Parallelly, we also set 
\[
   \nabla_{\hat{z}} = \nabla - \hat{z} \operatorname{d} \! z, \qquad  \theta_{\hat{z}} = \theta - \frac 12 \hat{z} \operatorname{d} \! z.
\]
We endow $X\setminus D$ by a Riemannian metric that is equal to the Poincar\'e metric\footnote{The set of $L^2$ forms of degree $1$ is independent on the conformal class of this metric.} near each point of $D$, and is smooth on the complement of $D$. 
We consider the Hermitian--Einstein metric $h$ on $E$. 

Let us define the analytic subsheaf of full rank (i.e., negative Hecke transform) $\mathcal{E}'\subset \mathcal{E}$ by setting, 
\begin{equation}\label{eq:E'}
    \mathcal{E}'(U) = \{ \mathbf{e} \in \mathcal{E}(U) \; \vert \quad \theta (\mathbf{e} ) \in \mathcal{E}\otimes K_{\mathbb{C}P^1} (U) \}
\end{equation}
for any analytic open $U\subset X$. 

\begin{lemma}
For every $\hat{z} \in \widehat{X} \setminus  \widehat{C}$ the $L^2$-cohomology space $L^2 H^1 (D_{\hat{z}})$ of degree $1$ of the elliptic complex 
\begin{equation}\label{eq:L2_complex}
     \Omega^{0}\otimes E \xrightarrow{D_{\hat{z}}}  \Omega^1 \otimes E \xrightarrow{D_{\hat{z}}} \Omega^2 \otimes E
\end{equation}
is canonically isomorphic to the hypercohomology 
\begin{equation}\label{eq:dR_complex}
     \mathbb{H}^1 \left(\mathcal{V} \xrightarrow{\nabla_{\hat{z}}} \mathcal{W} \otimes K_X(C) \right)
\end{equation}
for any pair of good lattices and to 
\begin{equation}\label{eq:Dolbeault_complex}
    \mathbb{H}^1 \left( \mathcal{E}' \xrightarrow{\theta_{\hat{z}}}  \mathcal{E}\otimes K_X \right). 
\end{equation}
\end{lemma}
\begin{proof}
By the K\"ahler identities, the Laplace operators associated to $D_{\hat{z}}$ and $\bar{\partial}_{\mathcal{E}} + \theta_{\hat{z}}$ only differ by the factor $2$.
In particular, their kernels agree. 
In different terms, the first $L^2$-cohomology space of~\eqref{eq:L2_complex} agrees with that of 
\begin{equation}\label{eq:Dolbeault_L2_complex}
     \Omega^{0}\otimes E \xrightarrow{\bar{\partial}_{\mathcal{E}} + \theta_{\hat{z}}}  \Omega^1 \otimes E \xrightarrow{\bar{\partial}_{\mathcal{E}} + \theta_{\hat{z}}} \Omega^2 \otimes E.
\end{equation}
We will only show that the first $L^2$-cohomology space of this latter complex is canonically isomorphic to~\eqref{eq:Dolbeault_complex}; the proof of the isomorphism between~\eqref{eq:L2_complex} and~\eqref{eq:dR_complex} is completely similar. 
We will only sketch the argument, see~\cite[Section~4.3]{Sz_Nahm}. 

First, we observe that there is a resolution
\begin{equation}\label{eq:Dolbeault_resolution}
   \mathcal{E}\otimes K_X \hookrightarrow L^2 (E\otimes \Omega^{1,0}) \xrightarrow{\bar{\partial}_{\mathcal{E}}} L^2 (E\otimes \Omega^2).     
\end{equation}
Indeed, let $\mathbf{e} \operatorname{d}\!t$ be a local meromorphic section of $E$ on a disc $B$ centered at $p\in D$ of sufficiently tiny radius, with holomorphic coordinate $t$ centered at $p$. 
Let us expand $\mathbf{e}$ with respect to a local holomorphic frame of $\mathcal{E}$ compatible with the parabolic structure: 
\[
    \mathbf{e} = f_1 \mathbf{e}_1 + f_2 \mathbf{e}_2 + f_3 \mathbf{e}_3, 
\]
for suitable meromorphic coefficients $f_i$ with poles at $p$. 
We need to show that $\mathbf{e}$ is square-integrable if and only the singularities of $f_1, f_2, f_3$ at $p$ are removable. 
The Hermitian--Einstein metric $h$ on $E$ is mutually bounded with~\eqref{eq:model_metric}, and the Riemannian metric $g$ on the base disc is the Poincar\'e metric 
\[
    \frac{|\!\operatorname{d}\!t |^2}{|t |^2 |\ln |t|^2|^2}. 
\]
Then, we have 
\[
    | \mathbf{e}_i |_{h,g}^2 \approx  |t|^{2 \alpha_i} |\ln |t|^2|^{k_i} , 
\]
where $\approx$ stands for mutually bounded. 
We deduce 
\begin{align*}
    \int_B  | \mathbf{e}  \operatorname{d}\!t |_{h,g}^2 \operatorname{d}\! Vol_g & \approx  \int_B  \sum_{i=1}^3 | f_i|^2  |t|^{2 \alpha_i} |\ln |t|^2|^{k_i} \frac{|\!\operatorname{d}\!t |^2}{|t |^2 |\ln |t|^2|^2} \\
    & = \int_B  \sum_{i=1}^3 | f_i|^2  |t|^{2 \alpha_i - 2} |\ln |t|^2|^{k_i-2} |\!\operatorname{d}\!t |^2 \\
    & = \sum_{i=1}^3  \int_B | f_i|^2  r^{2 \alpha_i - 1} |\ln |t|^2|^{k_i-2} |  \operatorname{d}\! r \operatorname{d}\! \varphi 
\end{align*}
where $t = r e^{\sqrt{-1}\varphi}$ are polar coordinates. 
Now, by our assumption~\eqref{eq:par_wt}, the integrals on the right hand side are bounded if $f_i$ are bounded. 
Conversely, if at least one of the $f_i$ has a pole, then the corresponding integral is unbounded. 
This shows that~\eqref{eq:Dolbeault_resolution} is exact at its middle term. 
Exactness at the last term follows from the $\bar{\partial}$-Poincar\'e lemma. 

Now, we let 
\begin{align*}
    \tilde{L}^2 (E ) & = \{ \mathbf{e} \in E(U) \; \vert \quad \theta (e) \in {L}^2 (E \otimes \Omega^{1,0}), \; \bar{\partial}_{\mathcal{E}} \mathbf{e} \in  {L}^2 (E \otimes \Omega^{0,1}) \} \\
    \tilde{L}^2 (E \otimes \Omega^{0,1}) & = \{ \mathbf{e} \operatorname{d}\!\bar{z} \in E\otimes \Omega^{0,1}(U) \; \vert \quad \theta \wedge e \in {L}^2 (E \otimes \Omega^2) \}
\end{align*}
We then obtain similarly the resolution 
\[
 \mathcal{E}' \hookrightarrow \tilde{L}^2 (E ) \xrightarrow{\bar{\partial}_{\mathcal{E}}} \tilde{L}^2 (E\otimes \Omega^{0,1}). 
\]
We deduce that the hypercohomology spaces of~\eqref{eq:Dolbeault_complex} are canonically isomorphic to the cohomology spaces of the double complex 
\[
\xymatrix{\tilde{L}^2 (E\otimes \Omega^{0,1}) \ar[r]^{\theta_{\hat{z}}} & L^2 (E\otimes \Omega^2) \\ 
 \tilde{L}^2 (E ) \ar[r]^{\theta_{\hat{z}}}  \ar[u]^{\bar{\partial}_{\mathcal{E}}} & L^2 (E\otimes \Omega^{1,0}) \ar[u]_{\bar{\partial}_{\mathcal{E}}}}. 
\]
This is, by its definition, the $L^2$-cohomology of~\eqref{eq:Dolbeault_L2_complex}.  
\end{proof}
We then identify $L^2 H^1 (D_{\hat{z}})$ with the space of $D_{\hat{z}}$-harmonic $1$-forms, and let 
\[
    \widehat{\pi}_{\hat{z}} \colon L^2 (X, E \otimes \Omega^1) \to L^2 H^1 (D_{\hat{z}})
\]
stand for orthogonal projection. 
Next, we define the transformed vector bundle of $E$ by 
\[
    \widehat{E}_{\hat{z}} =  L^2 H^1 (D_{\hat{z}}), 
\]
and we endow it with the Hermitian metric 
\[
    \hat{h} (\alpha, \beta ) = \int_X  h (\alpha, \beta )
\]
for any harmonic $E$-valued $1$-forms $\alpha, \beta$ on $X$ representing their respective $L^2$-cohomology classes. 
Here, if $\alpha = a_{1} \operatorname{d} \! z + a_{\bar{1}} \operatorname{d} \! \bar{z}$ and $\beta = b_{1} \operatorname{d} \! z + b_{\bar{1}} \operatorname{d} \! \bar{z}$ then 
\[
    h (\alpha, \beta ) = \left( h(a_{1}, b_{1}) - h (a_{\bar{1}}, b_{\bar{1}}) \right) \operatorname{d} \! z \wedge \operatorname{d} \! \bar{z}.
\]
Moreover, we endow $\widehat{E}$ by the holomorphic structure defined by the $\bar{\partial}$-operator given by 
\[
    \widehat{\pi} \circ \frac{\partial}{\partial \overline{\hat{z}}},
\]
where the latter component denotes the trivial $(0,1)$-connection with respect to the variable $\hat{z}$ on the $L^2$-sections of $E$ on $X$, viewed as a trivial Hilbert bundle on $\widehat{X}$. 
We denote by $\widehat{\mathcal{E}}$ the resulting holomorphic vector bundle. 
We extend $\widehat{\mathcal{E}}$ over $\widehat{C}$ by the formula~\eqref{eq:Dolbeault_complex}. 
We define a parabolic structure on $\widehat{\mathcal{E}}$ at $\widehat{C}$ by the locally free subsheaves 
\[
    \widehat{\mathcal{E}}_{\alpha } = \mathbb{H}^1 \left( \mathcal{E}'_{\alpha }\xrightarrow{\theta_{\hat{z}}}  \mathcal{E}_{\alpha }\otimes K_X (D) \right), 
\]
where 
\[
    \mathcal{E}'_{\alpha } (U) = \{ e\in \mathcal{E}_{\alpha }(U) \; \vert \quad \theta (e) \in \mathcal{E}_{\alpha }\otimes K_{\mathbb{C}P^1} (U) \}.
\]
Finally, we define the transformed Higgs field by 
\[
    ( \widehat{\theta} (\alpha ))(\hat{z}) = -\frac 12 \widehat{\pi}_{\hat{z}} (z \alpha ) \operatorname{d} \! \hat{z}. 
\]
This way, $(\widehat{\mathcal{E}}, \hat{h}, \widehat{\theta})$ is a wild harmonic bundle, and $\mathcal{F}, \mathcal{N}$ are merely the de Rham, respectively Dolbeault interpretation of the transformation of harmonic bundles 
\[
    (\mathcal{E}, h, \theta ) \mapsto (\widehat{\mathcal{E}}, \hat{h}, \widehat{\theta}). 
\]
\end{proof}

\subsection{Moduli spaces of sheaves on $Z^*$}

We now consider the log-Calabi--Yau pairs $(Z^*, F_{\infty}^{*})$ described in Section~\ref{sec:Dol}. 
The line bundle associated to $F_{\infty}^{*}$ is thus $K_{Z^*}^{-1}$, see Proposition \ref{prop:finf}. 
By Lemma~\ref{lem:pencil}, we have 
\[
    H^0 (Z^*, K_{Z^*}^{-1}) \cong \mathbb{C}^2.
\]
We let $\mathcal{M}_{\textrm{Muk}}^{*}$ denote the moduli space~\cite[Example~0.5]{Muk} of simple coherent sheaves $\mathcal{S}$ on $Z^*$ such that 
\begin{itemize}
    \item $c_0 (\mathcal{S} ) = 0\in H^0 (Z^*, \mathbb{Z})$, 
    \item $\operatorname{supp} (\mathcal{S}) \approx F_{\infty}^{*}$, 
    \item $\int_{\operatorname{supp} (\mathcal{S})} c_1 ( \mathcal{S}) = d+3$, 
\end{itemize}
where $\operatorname{supp}(\mathcal{S})$ stands for the support of $\mathcal{S}$, $\approx$ means linear equivalence and $d = \deg \mathcal{E}$. %$[pt]\in H_0 (Z^*, \mathbb{Z})$ is the homology class of a point, 
%Mukai's original notation for this space is $M(0,l,1)$, where we chose $l$ to be the anticanonical class, so this refers to sheaves of generic rank $0$ (i.e., supported in dimension $\leq 1$), supported on an anticanonical curve, and of Euler characteristic $1$. %TODO:check Euler char!
It is a smooth symplectic quasi-projective surface; let us denote by 
\[
    \mathbf{\Omega}_{\textrm{Muk}}^{*} \in H^{2,0} ( \mathcal{M}_{\textrm{Muk}}^{*})
\]
its holomorphic symplectic form. 
The Zariski tangent space at $\mathcal{S}\in \mathcal{M}_{\textrm{Muk}}^{*}$ is canonically isomorphic to the vector space $\operatorname{Ext}_{\mathcal{O}_{Z^*}}^1 (\mathcal{S} , \mathcal{S} )$. 
\begin{lemma}
Assume that the algebraic curve $\operatorname{supp} (\mathcal{S})$ is smooth. 
Then, $\operatorname{Ext}_{\mathcal{O}_{Z^*}}^2 (\mathcal{S} , \mathcal{S} ) \cong \mathbb{C}$. 
\end{lemma}
\begin{proof}
We follow~\cite[Remark~8.17]{DonMark} and~\cite[Lemma~5.3]{Sz_Plancherel}. 
Let us denote the support of $\mathcal{S}$ by $\Sigma$, the inclusion $\Sigma\subset Z^*$ by $\iota$, and the normal bundle of $\Sigma$ in $Z^*$ by $N_{\Sigma \subset Z^*}$. 
By the simplicity and smoothness assumptions, $\iota^* \mathcal{S}$ is an invertible sheaf on $\Sigma$. 
We have 
\begin{align*}
    \mathcal{E}xt^1_{\mathcal{O}_{Z^*}} (\mathcal{S} , \mathcal{S} ) & \cong \mathcal{E}xt^1_{\mathcal{O}_{Z^*}} (\mathcal{O}_{\Sigma} , \mathcal{O}_{\Sigma} ) \\
    & \cong \iota_* K_{\Sigma} \otimes_{\mathcal{O}_{Z^*}} K_{Z^*}^{-1} \\ 
    & \cong \iota_* N_{\Sigma \subset Z^*}. 
\end{align*}
On the other hand, the canonical Liouville symplectic structure of $Z^* \setminus F_{\infty}^{*}$ (for which $\Sigma\subset Z^*$ is Lagrangian) gives an isomorphism $N_{\Sigma \subset Z^*}\cong K_{\Sigma}$. 
It follows from the degeneration at $E_2$ of the spectral sequence abutting to $\operatorname{Ext}^2$ that  
\begin{align*}
        \operatorname{Ext}_{\mathcal{O}_{Z^*}}^2 (\mathcal{S} , \mathcal{S} ) \notag & \cong H^1 (Z^*, \mathcal{E}xt^1_{\mathcal{O}_{Z^*}} (\mathcal{S} , \mathcal{S} ) ) \\
        & \cong H^1 (\Sigma , N_{\Sigma \subset Z^*}) \\
        & \cong H^1 (\Sigma, K_{\Sigma} )\\
%        & \cong \operatorname{Ext}_{\mathcal{O}_{Z^*}}^2 (\mathcal{S} , \mathcal{S}\otimes \mathcal{O}(F_{\infty}^{*}) ) \notag \\
%        & \to \operatorname{Ext}_{\mathcal{O}_{Z^*}}^2 ( \mathcal{O}_{Z^*}, \mathcal{O}_{Z^*} ) \notag \\
%        & \cong H^0 ( Z^* , K_{Z^*} )^{\vee} \notag \\
%        & = H^0 ( Z^* , \mathcal{O}_{Z^*} (F_{\infty}^{*}) )^{\vee} \notag \\
        & \cong \mathbb{C},
\end{align*}
the last morphism being Serre duality. 
\end{proof}
We are now ready to give a concrete description of $\mathbf{\Omega}_{\textrm{Muk}}^{*}$: let its restriction to $T_{\mathcal{S}}\mathcal{M}_{\textrm{Muk}}^{*}$ be defined by the Yoneda product
\begin{equation}\label{eq:Mukai_form}
    \operatorname{Ext}_{\mathcal{O}_{Z^*}}^1 (\mathcal{S} , \mathcal{S} ) \times \operatorname{Ext}_{\mathcal{O}_{Z^*}}^1 (\mathcal{S} , \mathcal{S} ) \to \operatorname{Ext}_{\mathcal{O}_{Z^*}}^2 (\mathcal{S} , \mathcal{S} ) .
\end{equation}
% Here, the first morphism is Yoneda product, the second one is the dual of the natural morphism 
% \[
% \operatorname{Hom}_{\mathcal{O}_{Z^*}} ( \mathcal{O}_{Z^*}, \mathcal{O}_{Z^*} ) \to \operatorname{End}_{\mathcal{O}_{Z^*}} (\mathcal{S} )
% \]
% (that is an isomorphism by the definition of simplicity), the third one is Serre duality and the fourth one is 
% \[
%     \lambda_1 s_1 + \lambda_2 s_2 \mapsto \lambda_2, 
% \]
% see~\eqref{pencil}. 

\subsection{Comparison of natural symplectic structures}

We have seen in Theorem~\ref{thm:BNR} that there is an algebraic isomoprhism between $\mathcal{M}_{\textrm{Muk}}^{*}$ and a Zariski open subset of $\mathcal{M}_{\textrm{Dol}}^{JKT*}$. 
To relate the holomorphic symplectic structures of $\mathcal{M}_{\textrm{Dol}}^{JKT*}$ and $\mathcal{M}_{\textrm{Dol}}^{P*}$, we show that: 

\begin{prop}\label{prop:symplectic}
The isomorphism between $\mathcal{M}_{\textrm{Dol}}^{JKT*}$ and $\mathcal{M}_{\textrm{Muk}}^{*}$ preserves their holomorphic symplectic $2$-forms $\mathbf{\Omega}_{\textrm{Dol}}^{JKT*}$ and $\mathbf{\Omega}_{\textrm{Muk}}^{*}$, and the same holds for the isomorphism between $\mathcal{M}_{\textrm{Dol}}^{P*}$ and $\mathcal{M}_{\textrm{Muk}}^{*}$ too. 
\end{prop}

\begin{proof}
Our argument follows the lines of~\cite[Proposition~5.1]{Sz_Plancherel}. 
Clearly, it is sufficient to prove the statement for $\mathcal{M}_{\textrm{Dol}}^{JKT*}$, because for $\mathcal{M}_{\textrm{Dol}}^{P*}$ it follows the same argument. 
Moreover, it is sufficient to prove it on the Zariski open subset of $\mathcal{M}_{\textrm{Dol}}^{JKT*}$ defined by the property that $\Sigma_{(\mathcal{E},\theta)}$ is smooth. 

% Let $\widetilde{\mathcal{M}}_{\textrm{Dol}}^{JKT*}$ stand for the holomorphic Poisson moduli spaces of meromorphic Higgs bundles with divisor $D$, without fixing the parameters $a_i, b_i, c_i$. 
% Then, the spaces $\mathcal{M}_{\textrm{Dol}}^{JKT*}$ (where we fix the parameters) are symplectic leaves of $\widetilde{\mathcal{M}}_{\textrm{Dol}}^{JKT*}$. 
% Similarly, let $\widetilde{\mathcal{M}}_{\textrm{Muk}}^{*}$ stand for the moduli spaces of coherent sheaves on $\mathbb{F}_1$ whose support is a purely $1$-dimensional triple cover of $\mathbb{C}P^1$. 
% \begin{lemma}
% There exists an isomorphism 
% \[
%     \operatorname{Ext}^{\bullet} \left( (\mathcal{E}, \theta), (\mathcal{E}, \theta) \right) \cong  \operatorname{Ext}_{\mathcal{O}_{Z^*}}^{\bullet} (\mathcal{S}_{(\mathcal{E},\theta) } , \mathcal{S}_{(\mathcal{E},\theta) } )
% \]
% where $\operatorname{Ext}$ on the left-hand side is taken in the category of irregular Higgs bundles with fixed irregular type. 
% \end{lemma}
We will use the terminology and basic constructs of the theory of derived categories. 
For a smooth projective variety, let $D^b (Coh_Y )$ denote the bounded derived category of coherent sheaves of $\mathcal{O}_Y$-modules. 
Recall that the spectral sheaf $\mathcal{S}_{(\mathcal{E},\theta) }$ is defined by the projective resolution~\eqref{eq:spectral_sheaf}. 
In different terms, we have 
\[
    [p^* (\mathcal{E}  \otimes K_{\mathbb{C}P^1}(D)^{-1}) \xrightarrow{\zeta I_{\mathcal{E}}-\xi p^*\theta}  p^* \mathcal{E} \otimes \mathcal{O}_{\mathbb{F}_1|\mathbb{C}P^1}(1)] = [\mathcal{S}_{(\mathcal{E},\theta) }] \in D^b (Coh_{Z^*} ), 
\]
where the nontrivial sheaves of the complex on the left-hand side are placed in degrees $-1,0$. 
Given $(\mathcal{E},\theta)$, let $\mathcal{C}_{\bullet} = \mathcal{C}_{(\mathcal{E},\theta) }$ denote the image of this complex by the contravariant functor $\mathcal{H}om_{\mathcal{O}_{Z^*}}(\cdot , \mathcal{S}_{(\mathcal{E},\theta) } )$, i.e. define 
\begin{align}
    \mathcal{C}_0 & = \mathcal{H}om_{\mathcal{O}_{Z^*}} (p^* \mathcal{E} \otimes \mathcal{O}_{\mathbb{F}_1|\mathbb{C}P^1}(1), \mathcal{S}_{(\mathcal{E},\theta) })\label{eq:C0}\\
    \mathcal{C}_1 & = \mathcal{H}om_{\mathcal{O}_{Z^*}} (p^* (\mathcal{E}  \otimes K_{\mathbb{C}P^1}(D)^{-1}), \mathcal{S}_{(\mathcal{E},\theta) }) \notag \\
    & \cong  \mathcal{H}om_{\mathcal{O}_{Z^*}} (p^* \mathcal{E}, \mathcal{S}_{(\mathcal{E},\theta) }  \otimes p^* K_{\mathbb{C}P^1}(D)), \label{eq:C1}
\end{align}
with morphism $\mathcal{C}_0 \to \mathcal{C}_1$ given by 
\begin{equation}\label{eq:morphism}
        f \mapsto f \circ (\zeta I_{\mathcal{E}}-\xi p^*\theta).
\end{equation}
Then, we have 
\[
    \operatorname{Ext}_{\mathcal{O}_{Z^*}}^{\bullet} (\mathcal{S}_{(\mathcal{E},\theta) } , \mathcal{S}_{(\mathcal{E},\theta) } ) \cong \mathbb{H}^{\bullet} (\mathcal{C}_{(\mathcal{E},\theta)}). 
\]
\begin{lemma}
The first hypercohomology of $R^{\bullet} p_* \mathcal{C}_{\bullet}$ is canonically isomorphic to $T_{(\mathcal{E}, \theta)} \mathcal{M}_{\textrm{Dol}}^{JKT*}$. 
\end{lemma}
\begin{proof}
This follows immediately from Theorem~\ref{thm:BNR} by passing to tangent spaces. 
Let us, nevertheless, give some details. 

Since the complex is supported on $\Sigma_{(\mathcal{E},\theta)}$, and this scheme is finite over $X$, we have $R^{\bullet} p_* \mathcal{C}_{\bullet} = R^0 p_* \mathcal{C}_{\bullet}$. 
% By the projection formula, we have 
% \begin{equation*}
%     p_* \mathcal{C}_1 = \mathcal{H}om_{\mathcal{O}_{\mathbb{C}P^1}} (\mathcal{E}  \otimes K_{\mathbb{C}P^1}(D)^{-1}, \mathcal{E}) = \mathcal{E}nd_{\mathcal{O}_{\mathbb{C}P^1}} (\mathcal{E} ) \otimes K_{\mathbb{C}P^1}(D).
% \end{equation*}
% On the other hand, 
From the defining ring monomorphism $\mathcal{O}_{\mathbb{F}_1} \to \mathcal{O}_{Z^*}$ we get a natural morphism 
\[
    \mathcal{H}om_{\mathcal{O}_{\mathbb{F}_1}} (p^* \mathcal{E} \otimes \mathcal{O}_{\mathbb{F}_1|\mathbb{C}P^1}(1), \mathcal{S}_{(\mathcal{E},\theta) }) \to \mathcal{H}om_{\mathcal{O}_{Z^*}} (p^* \mathcal{E} \otimes \mathcal{O}_{\mathbb{F}_1|\mathbb{C}P^1}(1), \mathcal{S}_{(\mathcal{E},\theta) }). 
\]
It gives rise to a natural homomorphism from $\mathbb{H}^{\bullet} (\mathcal{C}_{(\mathcal{E},\theta)})$ to $\mathbb{H}^{\bullet}$ of the complex 
\begin{equation}\label{eq:def_complex_Poisson_Mukai}
        [ \mathcal{H}om_{\mathcal{O}_{\mathbb{F}_1}} (p^* \mathcal{E} \otimes \mathcal{O}_{\mathbb{F}_1|\mathbb{C}P^1}(1), \mathcal{S}_{(\mathcal{E},\theta) }) \to  \mathcal{H}om_{\mathcal{O}_{\mathcal{O}_{\mathbb{F}_1}}} (p^* \mathcal{E}, \mathcal{S}_{(\mathcal{E},\theta) }  \otimes p^* K_{\mathbb{C}P^1}(D))] 
\end{equation}
(sheaves placed in degrees $0,1$). 
The latter hypercohomology governs the deformation theory of the Poisson moduli spaces of sheaves of $\mathbb{F}_1$, that correspond to the Poisson moduli spaces of meromorphic Higgs bundles with non-fixed parameter values $a_i, b_i, c_i\in \mathbb{C}$. 
It is easy to see that the direct image under $p$ of~\eqref{eq:def_complex_Poisson_Mukai} is quasi-isomorphic to the complex $\mathcal{D}(\operatorname{ad}_{\theta})_{\bullet} \in D^b (Coh_{\mathbb{C}P^1})$ defined by 
\begin{equation}\label{eq:Dol_Poisson_complex}
        \mathcal{E}nd_{\mathcal{O}_{\mathbb{C}P^1}} (\mathcal{E} ) \to \mathcal{E}nd_{\mathcal{O}_{\mathbb{C}P^1}} (\mathcal{E} ) \otimes K_{\mathbb{C}P^1}(D) 
\end{equation}
with nontrivial sheaves in degrees $0,1$. 
Here, the only nontrivial boundary map is the direct image of~\eqref{eq:morphism}. 
Now, given that  
\[
p_* (\zeta\colon  \mathcal{S}_{(\mathcal{E},\theta)} \to \mathcal{S}_{(\mathcal{E},\theta)} \otimes p^* K_{\mathbb{C}P^1}(D)) = (\theta\colon \mathcal{E} \to \mathcal{E} \otimes K_{\mathbb{C}P^1}(D)), 
\]
we get that the boundary map of~\eqref{eq:Dol_Poisson_complex} is 
\[
    F \mapsto \theta \circ F - F \circ \theta.  
\]
% \[
%     p_* \mathcal{C}_0 \subset \mathcal{E}nd_{\mathcal{O}_{\mathbb{C}P^1}} (\mathcal{E} )
% \]
Therefore, there exists a natural morphism from  the direct image under $p$ of~\eqref{eq:def_complex_Poisson_Mukai} to $R^{\bullet} p_* \mathcal{C}_{\bullet}$. 
The induced morphism in first hypercohomology of this natural map identifies to the tangent map of the inclusion of $T_{(\mathcal{E}, \theta)} \mathcal{M}_{\textrm{Dol}}^{JKT*}$ in the tangent space of the Poisson moduli spaces of meromorphic Higgs bundles. 
\end{proof}
From the above properties and proper base change, we get the isomorphisms 
\begin{align*}
    T_{\mathcal{S}_{(\mathcal{E},\theta) }} \mathcal{M}_{\textrm{Muk}}^{*} & = \operatorname{Ext}_{\mathcal{O}_{Z^*}}^1 (\mathcal{S}_{(\mathcal{E},\theta) } , \mathcal{S}_{(\mathcal{E},\theta) } ) \\
    & = \mathbb{H}^1 (\mathcal{C}_{(\mathcal{E},\theta)}) \\
    & \stackrel{R^{\bullet} p_*}{\cong} \mathbb{H}^1 (R^{\bullet} p_*\mathcal{C}_{(\mathcal{E},\theta)}) \\
    & = T_{(\mathcal{E}, \theta)} \mathcal{M}_{\textrm{Dol}}^{JKT*}.
\end{align*}

We proceed to describing in detail the natural alternating $2$-forms on these tangent spaces, based on~\cite[Section~10]{HL}. 
We start by $\mathcal{M}_{\textrm{Muk}}^{*}$, where the $2$-form $\mathbf{\Omega}_{\textrm{Muk}}^{*}$ is defined by~\eqref{eq:Mukai_form}. 
Let $\mathbb{C}P^1 = \cup_{i\in I} U_i$ be an acyclic analytic open covering indexed by a finite ordered set $I$. 
It is possible to choose the covering so that $\Sigma_{(\mathcal{E},\theta)}$ is unramified over $U_i \cap U_j$ for all $i\neq j \in I$. 
We have $Z^* = \cup_i p^{-1}(U_i)$. 
We will use \v{C}ech cohomology with respect to these covers.
As customary, we will denote \v{C}ech $i$-cochains by $\mathcal{C}^i$; this is not to be confused with the complex defined in~\eqref{eq:C0}-\eqref{eq:C1}. 
An element of $\mathbb{H}^1 (\mathcal{C}_{(\mathcal{E},\theta)})$ is represented by 
\[
    ( (f_{ij})_{ij}, (g_i )_i ) \in \mathcal{C}^1 ( \mathcal{C}_0 ) \oplus \mathcal{C}^0 ( \mathcal{C}_1 ).
\]
Namely, 
\[
    f_{ij} \in \mathcal{H}om_{\mathcal{O}_{p^{-1}(U_i\cap U_j)}} (p^* \mathcal{E} \otimes \mathcal{O}_{\mathbb{F}_1|\mathbb{C}P^1}(1), \mathcal{S}_{(\mathcal{E},\theta) })
\]
and 
\[
    g_i \in \mathcal{H}om_{\mathcal{O}_{p^{-1}(U_i)}} (p^* (\mathcal{E} \otimes  K_{\mathbb{C}P^1}(D)^{-1}), \mathcal{S}_{(\mathcal{E},\theta) }  ), 
\]
and they fulfill the relation 
\[
    f_{ij} \circ (\zeta I_{\mathcal{E}}-\xi p^*\theta) = (g_i)|_{p^{-1}(U_i\cap U_j)} - (g_j)|_{p^{-1}(U_i\cap U_j)}.
\]
Let now 
\[
    ( (f'_{ij})_{ij}, (g'_i )_i )
\]
represent another element of $\mathbb{H}^1 (\mathcal{C}_{(\mathcal{E},\theta)})$. 
% We consider the sheaf of $\mathcal{O}_{\Sigma_{(\mathcal{E},\theta)}}$-modules 
% \[
%     p^* \mathcal{E} \otimes \mathcal{O}_{\mathbb{F}_1|\mathbb{C}P^1}(1) \to \mathcal{S}_{(\mathcal{E},\theta) } \to 0 
% \]
By the smoothness assumption on $\Sigma_{(\mathcal{E},\theta)}$, $\mathcal{S}_{(\mathcal{E},\theta) }$ is a projective $\mathcal{O}_{\Sigma_{(\mathcal{E},\theta)}}$-module. 
Therefore, for each $i$ the morphism of sheaves of $\mathcal{O}_{\Sigma_{(\mathcal{E},\theta)}}$-modules 
\[
    g_i\otimes_{\mathcal{O}_{Z^*}} \operatorname{I}_{\mathcal{O}_{\Sigma_{(\mathcal{E},\theta)}}} \colon p^* ( \mathcal{E} \otimes K_{\mathbb{C}P^1}(D)^{-1}) (p^{-1} (U_i ) \cap \Sigma_{(\mathcal{E},\theta)}) \to \mathcal{S}_{(\mathcal{E},\theta) } (p^{-1} (U_i )\cap \Sigma_{(\mathcal{E},\theta)}), 
\]
coming from tensor product of~\eqref{eq:spectral_sheaf} with  $\mathcal{O}_{\Sigma_{(\mathcal{E},\theta)}}$ over $\mathcal{O}_{Z^*}$, admits a lift 
\[
    \tilde{g}_i \colon p^* ( \mathcal{E} \otimes K_{\mathbb{C}P^1}(D)^{-1}) (p^{-1} (U_i )\cap \Sigma_{(\mathcal{E},\theta)}) \to p^* \mathcal{E} \otimes \mathcal{O}_{\mathbb{F}_1|\mathbb{C}P^1}(1)(p^{-1} (U_i )\cap \Sigma_{(\mathcal{E},\theta)}).
\]
Then, the Yoneda product of the classes $( (f_{ij})_{ij}, (g_i )_i ), ( (f'_{ij})_{ij}, (g'_i )_i )$ is the class in $\mathbb{H}^2 (\mathcal{C}_{(\mathcal{E},\theta)})$ represented by 
\[
     f_{ij} \circ (\widetilde{g'}_j)|_{p^{-1}(U_i\cap U_j)} - f'_{ij} \circ (\tilde{g}_i)|_{p^{-1}(U_i\cap U_j)} \in \mathcal{C}^1 ( \mathcal{C}_1),
\]
(see~\cite[Section~10.1.1]{HL}). 
% We have 
% \[
%     \mathcal{T}or_1^{\mathcal{O}_{Z^*}} ( \mathcal{S}_{(\mathcal{E},\theta) } , \mathcal{O}_{\Sigma_{(\mathcal{E},\theta)}} ) \cong \mathcal{S}_{(\mathcal{E},\theta) } \otimes K_{\Sigma_{(\mathcal{E},\theta)}}^{-1} 
% \]
Let $\Delta\subset \Sigma_{(\mathcal{E},\theta)}$ denote the ramification divisor of the morphism $p|_{\Sigma_{(\mathcal{E},\theta)}}$. 
By Riemann--Hurwitz, we have 
\[
    \operatorname{length}(\Delta ) = 3 \chi (\mathbb{C}P^1) - \chi ( \Sigma_{(\mathcal{E},\theta)} ) =  6 . 
\]
Then the tensor product of~\eqref{eq:spectral_sheaf} with $\mathcal{O}_{\Sigma_{(\mathcal{E},\theta)}}$ reads as 
\begin{align*}
    0 & \to  \mathcal{S}_{(\mathcal{E},\theta) } \otimes K_{\Sigma_{(\mathcal{E},\theta)}}(p^{-1}(D) - \Delta )^{-1} \to  \\
     \xrightarrow{\phi} p^* ( \mathcal{E} \otimes K_{\mathbb{C}P^1}(D)^{-1}) \otimes_{\mathcal{O}_{Z^*}} \mathcal{O}_{\Sigma_{(\mathcal{E},\theta)}} \to p^* \mathcal{E} \otimes_{\mathcal{O}_{Z^*}} \mathcal{O}_{\Sigma_{(\mathcal{E},\theta)}} & \to \mathcal{S}_{(\mathcal{E},\theta) } \to 0
\end{align*}
Applying the contravariant functor $\mathcal{H}om_{\mathcal{O}_{\Sigma_{(\mathcal{E},\theta)}}}(\cdot , \mathcal{S}_{(\mathcal{E},\theta) } )$ to the morphism $\phi$ we get a $1$-cocycle 
\begin{equation}\label{eq:1_cocycle}
        f_{ij} \circ (\widetilde{g'}_j) \circ \phi |_{p^{-1}(U_i\cap U_j)} - f'_{ij} \circ (\tilde{g}_i) \circ \phi |_{p^{-1}(U_i\cap U_j)} 
\end{equation}
with values in 
\[
    \mathcal{H}om_{\mathcal{O}_{\Sigma_{(\mathcal{E},\theta)}}}( \mathcal{S}_{(\mathcal{E},\theta) } \otimes K_{\Sigma_{(\mathcal{E},\theta)}}(p^{-1}(D) - \Delta )^{-1} , \mathcal{S}_{(\mathcal{E},\theta) } ) \cong \mathcal{E}nd_{\mathcal{O}_{\Sigma_{(\mathcal{E},\theta)}}}( \mathcal{S}_{(\mathcal{E},\theta) }) \otimes K_{\Sigma_{(\mathcal{E},\theta)}}(p^{-1}(D) - \Delta ).
\]
Since $\mathcal{S}_{(\mathcal{E},\theta) }$ is locally free as an $\mathcal{O}_{\Sigma_{(\mathcal{E},\theta)}}$-module, this amounts to a $1$-cocycle with values in $K_{\Sigma_{(\mathcal{E},\theta)}}(p^{-1}(D) - \Delta )$. 
% Now, again by projectivity of $\mathcal{S}_{(\mathcal{E},\theta) }$ over $\Sigma_{(\mathcal{E},\theta)}$, this homomorphism admits a lift 
% \[
%     \widetilde{f_{ij} \widetilde{g'}_j} - \widetilde{f'_{ij}\tilde{g}_i} \colon p^* ( \mathcal{E} \otimes K_{\mathbb{C}P^1}(D)^{-1}) \to  p^*  \mathcal{E} \otimes \mathcal{O}_{\mathbb{F}_1|\mathbb{C}P^1}(1) 
% \]
% defined over $p^{-1} (U_i \cap U_j ) \cap \Sigma_{(\mathcal{E},\theta)}$. 
% Applying the contravariant functor  $\mathcal{H}om_{\mathcal{O}_{\Sigma_{(\mathcal{E},\theta)}}}(\cdot , \mathcal{S}_{(\mathcal{E},\theta) } )$
%  we get a $p^* K_{\mathbb{C}P^1}(D)$-valued $1$-cocycle on $\Sigma_{(\mathcal{E},\theta)}$, and thus a $K_{\Sigma_{(\mathcal{E},\theta)}}$-valued $1$-cocycle on $\Sigma_{(\mathcal{E},\theta)}$. 

\begin{lemma}\label{lem:1_cocycle}
The $1$-cocycle~\eqref{eq:1_cocycle} takes its values in $K_{\Sigma_{(\mathcal{E},\theta)}} ( - \Delta  )$. 
\end{lemma}
\begin{proof}
The constructions of Section~\ref{sec:DolVI}--~\ref{sec:DolI} are resolutions of the indeterminacy loci of the pencils in $\mathbb{F}_1$. 
In particular, for each $(\mathcal{E},\theta)$ the spectral curve $\Sigma_{(\mathcal{E},\theta)}$ is disjoint from (the total transform in $Z^*$ of) $p^{-1}(D)$. 
Said differently, there exists a canonical isomorphism 
\[
    \mathcal{O}_{\Sigma_{(\mathcal{E},\theta)}} \otimes \mathcal{O}_{Z^*}(p^{-1}(D)) \cong \mathcal{O}_{\Sigma_{(\mathcal{E},\theta)}}.
\]
\end{proof}
By Serre duality, we have 
\[
    H^1 ( \Sigma_{(\mathcal{E},\theta)}, K_{\Sigma_{(\mathcal{E},\theta)}}( - \Delta  )) \cong H^0 ( \Sigma_{(\mathcal{E},\theta)}, \mathcal{O}_{\Sigma_{(\mathcal{E},\theta)}}( \Delta  ))^{\vee}. 
\]
There is a canonical embedding 
\[
    H^0 ( \Sigma_{(\mathcal{E},\theta)}, \mathcal{O}_{\Sigma_{(\mathcal{E},\theta)}}) \hookrightarrow H^0 ( \Sigma_{(\mathcal{E},\theta)}, \mathcal{O}_{\Sigma_{(\mathcal{E},\theta)}}( \Delta  )), 
\]
so an element of $H^1 ( \Sigma_{(\mathcal{E},\theta)}, K_{\Sigma_{(\mathcal{E},\theta)}}( - \Delta  ))$ may be evaluated on its image. 
Therefore, this procedure provides the desired pairing 
\begin{equation}\label{eq:Mukai_pairing}
      \mathbf{\Omega}_{\textrm{Muk}}^{*}( [ (f_{ij})_{ij}, (g_i )_i ], [ (f'_{ij})_{ij}, (g'_i )_i ]) = \langle [ (f_{ij} \circ (\widetilde{g'}_j) \circ \phi  - f'_{ij} \circ (\tilde{g}_i) \circ \phi)_{ij} ] , 1  \rangle \in \mathbb{C}. 
\end{equation}

We now describe the symplectic form $\mathbf{\Omega}_{\textrm{Dol}}^{JKT*}$. 
For reference, see~\cite[Section~4]{BR}. 
Let $(R^{\bullet} p_* \mathcal{C} \otimes R^{\bullet} p_* \mathcal{C})_{\bullet} \in D^b (Coh_{\mathbb{C}P^1})$ be defined by 
\begin{align*}
    (R^{\bullet} p_* \mathcal{C} \otimes R^{\bullet} p_* \mathcal{C})_0 & = R^{\bullet} p_* \mathcal{C}_0 \otimes R^{\bullet} p_* \mathcal{C}_0 \\
    (R^{\bullet} p_* \mathcal{C} \otimes R^{\bullet} p_* \mathcal{C})_1 & = (R^{\bullet} p_* \mathcal{C}_1 \otimes R^{\bullet} p_* \mathcal{C}_0 ) \oplus (R^{\bullet} p_* \mathcal{C}_0 \otimes R^{\bullet} p_* \mathcal{C}_1 ) \\
    (R^{\bullet} p_* \mathcal{C} \otimes R^{\bullet} p_* \mathcal{C})_2 & = R^{\bullet} p_* \mathcal{C}_1 \otimes R^{\bullet} p_* \mathcal{C}_1, 
\end{align*}
with differentials 
\begin{align*}
    (\operatorname{ad}_{\theta}\otimes \mbox{I})  \oplus  (\mbox{I}\otimes \operatorname{ad}_{\theta}) & \colon (R^{\bullet} p_* \mathcal{C} \otimes R^{\bullet} p_* \mathcal{C})_0 \to (R^{\bullet} p_* \mathcal{C} \otimes R^{\bullet} p_* \mathcal{C})_1 \\
    (\mbox{I}\otimes \operatorname{ad}_{\theta}) - (\operatorname{ad}_{\theta}\otimes \mbox{I}) & \colon (R^{\bullet} p_* \mathcal{C} \otimes R^{\bullet} p_* \mathcal{C})_1 \to (R^{\bullet} p_* \mathcal{C} \otimes R^{\bullet} p_* \mathcal{C})_2.
\end{align*}
There exists a natural cup product 
\[
    \cup\colon \mathbb{H}^1 ( R^{\bullet} p_* \mathcal{C}_{\bullet} ) \times \mathbb{H}^1 ( R^{\bullet} p_* \mathcal{C}_{\bullet} ) \to \mathbb{H}^2 ( (R^{\bullet} p_* \mathcal{C} \otimes R^{\bullet} p_* \mathcal{C})_{\bullet} ). 
\]
A class in $\mathbb{H}^1 ( R^{\bullet} p_* \mathcal{C}_{\bullet} )$ is represented by 
\[
   ( (F_{ij})_{ij}, (G_i )_i ) \in \mathcal{C}^1 ( p_* \mathcal{C}_0 ) \oplus \mathcal{C}^0 ( p_* \mathcal{C}_1 ), 
\]
i.e. 
\[
    F_{ij} \in \mathcal{E}nd_{\mathcal{O}_{U_i\cap U_j}} ( \mathcal{E})
\]
and 
\[
    G_i \in \mathcal{H}om_{\mathcal{O}_{U_i}} (\mathcal{E}, \mathcal{E} \otimes  K_{\mathbb{C}P^1}(D)), 
\]
and they fulfill the relation 
\[
    \operatorname{ad}_{\theta} (F_{ij}) = (G_i)|_{U_i\cap U_j} - (G_j)|_{U_i\cap U_j}.
\]
Let us consider the morphism 
\begin{equation}\label{eq:trace_morphism}
        (R^{\bullet} p_* \mathcal{C} \otimes R^{\bullet} p_* \mathcal{C})_{\bullet} \to K_{\mathbb{C}P^1} (D) [-1]
\end{equation}
in $D^b (Coh_{\mathbb{C}P^1})$ defined by the sheaf morphism 
\begin{align*}
    (R^{\bullet} p_* \mathcal{C} \otimes R^{\bullet} p_* \mathcal{C})_1 & \to K_{\mathbb{C}P^1} (D) \\
    ((G \otimes F) , (F' \otimes G')) & \mapsto \operatorname{tr} (G F) - \operatorname{tr} (G' F'), 
\end{align*}
and all other morhpisms $0$. 
\begin{lemma}
    The morphism~\eqref{eq:trace_morphism} maps $(R^{\bullet} p_* \mathcal{C} \otimes R^{\bullet} p_* \mathcal{C})_{\bullet}$ into $K_{\mathbb{C}P^1}[-1]$. 
\end{lemma}
\begin{proof}
Immediate by taking the direct image of Lemma~\ref{lem:1_cocycle}.
\end{proof}
We deduce a natural morphism on hypercohomology 
\[
    \mathbb{H}^2 ( (R^{\bullet} p_* \mathcal{C} \otimes R^{\bullet} p_* \mathcal{C})_{\bullet} ) \to \mathbb{H}^2 ( K_{\mathbb{C}P^1} [-1] ) = H^1 (\mathbb{C}P^1, K_{\mathbb{C}P^1} ). 
\]
The symplectic form $\mathbf{\Omega}_{\textrm{Dol}}^{JKT*}$ is then defined by the composition 
\begin{align*}
     \mathbb{H}^1 ( R^{\bullet} p_* \mathcal{C}_{\bullet} ) \times \mathbb{H}^1 ( R^{\bullet} p_* \mathcal{C}_{\bullet} ) & \xrightarrow{\cup} \mathbb{H}^2 ((R^{\bullet} p_* \mathcal{C} \otimes R^{\bullet} p_* \mathcal{C})_{\bullet} ) \\ 
     & \to H^1 (\mathbb{C}P^1, K_{\mathbb{C}P^1}) \\
     & \xrightarrow{SD} H^0 (\mathbb{C}P^1, \mathcal{O}_{\mathbb{C}P^1} )^{\vee} \cong \mathbb{C}, 
\end{align*}
where $SD$ stands for Serre duality. 
Let now 
\[
    ((F_{ij})_{ij}, (G_i )_i), \quad ((F'_{ij})_{ij}, (G'_i )_i)
\]
be arbitrary \v{C}ech cocycles representing elements in $\mathbb{H}^1 ( R^{\bullet} p_* \mathcal{C}_{\bullet} )$. 
Their cup product is represented by the $1$-cocycle 
\[
   ((F_{ij} \otimes G'_j |_{U_i \cap U_j}, - G_i \otimes F_{ij}' |_{U_i \cap U_j})_{ij}, (G_i \otimes G'_i - G'_i \otimes G_i)_i)
\]
where 
\begin{align*}
 (F_{ij} \otimes G'_j |_{U_i \cap U_j}, - G_i \otimes F_{ij}'  |_{U_i \cap U_j})_{ij} & \in \mathcal{C}^1 ((R^{\bullet} p_* \mathcal{C} \otimes R^{\bullet} p_* \mathcal{C})_1) \\
  (G_i \otimes G'_i - G'_i \otimes G_i)_i & \in \mathcal{C}^0 ((R^{\bullet} p_* \mathcal{C} \otimes R^{\bullet} p_* \mathcal{C})_2). 
\end{align*}
Applying~\eqref{eq:trace_morphism}, we get the $K_{\mathbb{C}P^1}$-valued $1$-cocycle 
\[
    \operatorname{tr} ( G'_j |_{U_i \cap U_j} F_{ij} - G_i |_{U_i \cap U_j} F_{ij}'). 
\]
Let us now take 
\[
    F_{ij} = p_* f_{ij}, \quad G_i = p_* g_i, \quad F'_{ij} = p_* f'_{ij}, \quad G'_i = p_* g'_i. 
\]
By~\eqref{eq:Mukai_pairing}, we then need to prove 
\[
   [ \operatorname{tr} ( G'_j |_{U_i \cap U_j} F_{ij} - G_i |_{U_i \cap U_j} F_{ij}')_{ij} ] = (p^*)^t [ (f_{ij} \circ (\widetilde{g'}_j) \circ \phi  - f'_{ij} \circ (\tilde{g}_i) \circ \phi)_{ij}  ] \in H^1 ( \mathbb{C}P^1, K_{\mathbb{C}P^1} ).  
\]
Here, 
\[
    p^* \colon H^0 (\mathbb{C}P^1, \mathcal{O}_{\mathbb{C}P^1} ) \to H^0 ( \Sigma_{(\mathcal{E},\theta)}, \mathcal{O}_{\Sigma_{(\mathcal{E},\theta)}} )
\]
is the pull-back map and 
\[
    (p^*)^t\colon H^1 ( \Sigma_{(\mathcal{E},\theta)}, K_{\Sigma_{(\mathcal{E},\theta)}} ) \to H^1 ( \mathbb{C}P^1, K_{\mathbb{C}P^1} )
\]
stands for its transpose via Serre duality. Let 
\[
    p^{-1} (U_i \cap U_j)_m \qquad \mbox{for} \; 1 \leq m \leq 3
\]
be some fixed ordering of the sheets of $\Sigma_{(\mathcal{E},\theta)}$ lying over $U_i \cap U_j$ (recall our assumption $\Delta \cap p^{-1} (U_i \cap U_j ) = \varnothing$ for any $i\neq j$). 
The morphism $\phi$ appearing in~\eqref{eq:1_cocycle} induces isomorphisms 
\[
    \mathcal{E}|_{U_i \cap U_j} \cong \bigoplus_{m=1}^3 \mathcal{S}_{(\mathcal{E},\theta) }|_{p^{-1} (U_i \cap U_j)_m}. 
\]
We have 
\[
    \operatorname{tr} ( G'_j |_{U_i \cap U_j} F_{ij} ) = \operatorname{tr} (F_{ij} G'_j |_{U_i \cap U_j}) =  \sum_{m=1}^3 \left( f_{ij} \circ (\widetilde{g'}_j) \circ \phi \right)|_{p^{-1}(U_i\cap U_j)_m}, 
\]
and similarly for $\operatorname{tr} ( G_i |_{U_i \cap U_j} F'_{ij} )$. 
Now, since 
\[
    p^{-1} (U_i \cap U_j) = \bigcup_{m = 1}^3 p^{-1} (U_i \cap U_j)_m 
\]
(disjoint union), the set of double intersections of the open subsets $p^{-1}(U_i)$ is just 
\[
    \{ p^{-1} (U_i \cap U_j)_m \}_{i,j\in I, 1 \leq m \leq 3}. 
\]
We get 
\[
    (\operatorname{tr} ( G'_j |_{U_i \cap U_j} F_{ij} - G_i |_{U_i \cap U_j} F_{ij}'))_{ij} = \left( \sum_{m = 1}^3  \left( f_{ij} \circ (\widetilde{g'}_j) \circ \phi - f'_{ij} \circ (\widetilde{g}_i) \circ \phi \right) |_{p^{-1}(U_i\cap U_j)_m} \right)_{i,j\in I}.  
\]
This shows that 
\[
    \xymatrix{\operatorname{Ext}_{\mathcal{O}_{Z^*}}^1 (\mathcal{S} , \mathcal{S} ) \times \operatorname{Ext}_{\mathcal{O}_{Z^*}}^1 (\mathcal{S} , \mathcal{S} ) \ar[r]^{\hspace{1.3cm}\cup} \ar[d]_{p_* \times p_*} &  \operatorname{Ext}_{\mathcal{O}_{Z^*}}^2 (\mathcal{S} , \mathcal{S} ) \ar[r] &  H^1 ( \Sigma_{(\mathcal{E},\theta)}, K_{\Sigma_{(\mathcal{E},\theta)}}) \ar[d]^{(p^*)^t}\\
    \mathbb{H}^1 ( R^{\bullet} p_* \mathcal{C}_{\bullet} ) \times \mathbb{H}^1 ( R^{\bullet} p_* \mathcal{C}_{\bullet} ) \ar[r]^{\cup} &  \mathbb{H}^2 ((R^{\bullet} p_* \mathcal{C} \otimes R^{\bullet} p_* \mathcal{C})_{\bullet} ) \ar[r] & H^1 (\mathbb{C}P^1, K_{\mathbb{C}P^1}) }    
\]
commutes. 
It follows that 
\[
    \mathbf{\Omega}_{\textrm{Dol}}^{JKT*} ([ (F_{ij})_{ij}, (G_i )_i ], [ (F'_{ij})_{ij}, (G'_i )_i ]) = \mathbf{\Omega}_{\textrm{Muk}}^{*}( [ (f_{ij})_{ij}, (g_i )_i ], [ (f'_{ij})_{ij}, (g'_i )_i ]). 
\]
This finishes the proof.  
\end{proof}

The proof of Theorem~\ref{thm:Dolisometry} now follows from the proposition, because on Zariski open subsets of $\mathcal{M}_{\textrm{Dol}}^{JKT*}$ and $\mathcal{M}_{\textrm{Dol}}^{P*}$ the holomorphic symplectic $2$-forms get mapped to each other by $\mathcal{N}$.

\bigskip
\textbf{Funding}
\bigskip
\\ The project supported by the Doctoral Excellence Fellowship Programme (DCEP) is funded by the National Research Development and Innovation Fund of the Ministry of Culture and Innovation and the Budapest University of Technology and Economics, under a grant agreement with the National Research, Development and Innovation Office. During the preparation of this manuscript, the authors were supported by the grant K146401 of the National Research, Development and Innovation Office.
The second author was also supported by the grant KKP144148 of the National Research, Development and Innovation Office.

\end{document}